\documentclass[a4paper, 11pt, oneside]{amsart}

\title{Semistability and asymptotics of geometric flows}
\date{\today}
\author{Shing Tak Lam}
\address{Mathematics Institute, University of Warwick, Coventry CV4 7AL, United Kingdom}
\email{Shing-Tak.Lam@warwick.ac.uk}

\usepackage[margin=2.5cm]{geometry}

\usepackage[UKenglish]{babel}
\usepackage{csquotes}

\usepackage[
    style=alphabetic, 
    isbn=false, 
    doi=false,
    url=false,
    eprint=false,
    maxalphanames=4,
    maxbibnames=4]{biblatex}
\usepackage[ocgcolorlinks, linkcolor=blue, citecolor=cyan]{hyperref}

\usepackage{tikz-cd}

\usepackage{amssymb}
\usepackage{physics}
\usepackage{mathtools}
\usepackage{enumerate}
\usepackage[shortlabels]{enumitem}
\usepackage{caption}
\usepackage{hypcap}

\newcommand{\N}{\mathbb N}
\newcommand{\Z}{\mathbb Z}
\newcommand{\Q}{\mathbb Q}
\newcommand{\R}{\mathbb R}
\newcommand{\C}{\mathbb C}
\renewcommand{\P}{\mathbb P}

\newcommand{\FS}{\mathrm{FS}}

\renewcommand{\grad}{\nabla}
\renewcommand{\epsilon}{\varepsilon}

\DeclareMathOperator{\Ad}{Ad}
\DeclareMathOperator{\pr}{pr}

\DeclareMathOperator{\Rep}{Rep}
\DeclareMathOperator{\GL}{GL}
\DeclareMathOperator{\gr}{gr}
\DeclareMathOperator{\Hom}{Hom}
\DeclareMathOperator{\End}{End}
\DeclareMathOperator{\Aut}{Aut}
\DeclareMathOperator{\id}{id}
\DeclareMathOperator{\Lie}{Lie}
\DeclareMathOperator{\Ham}{Ham}
\DeclareMathOperator{\Cal}{Cal}
\DeclareMathOperator{\Stab}{Stab}

\DeclareMathOperator{\im}{im}
\DeclareMathOperator{\ad}{ad}

\DeclareMathOperator{\Bl}{Bl}

\DeclareMathOperator{\Coh}{Coh}

\DeclareMathOperator{\Met}{Met}

\newcommand{\mfg}{\mathfrak g}
\newcommand{\mfk}{\mathfrak k}
\newcommand{\mfm}{\mathfrak m}
\newcommand{\mft}{\mathfrak t}

\newcommand{\mcA}{\mathcal A}

\newcommand{\mcC}{\mathcal C}
\newcommand{\mcD}{\mathcal D}
\newcommand{\mcE}{\mathcal E}
\newcommand{\mcF}{\mathcal F}
\newcommand{\mcG}{\mathcal G}
\newcommand{\mcH}{\mathcal H}

\newcommand{\mcJ}{\mathcal J}
\newcommand{\mcK}{\mathcal K}
\newcommand{\mcM}{\mathcal M}
\newcommand{\mcN}{\mathcal N}
\newcommand{\mcO}{\mathcal O}
\newcommand{\mcQ}{\mathcal Q}

\newcommand{\mcS}{\mathcal S}
\newcommand{\mcT}{\mathcal T}
\newcommand{\mcU}{\mathcal U}
\newcommand{\mcX}{\mathcal X}
\newcommand{\mcZ}{\mathcal Z}

\renewcommand{\ss}{\mathrm{ss}}
\newcommand{\ps}{\mathrm{ps}}
\newcommand{\st}{\mathrm{s}}

\newcommand{\del}{\partial}
\newcommand{\delbar}{\overline{\partial}}

\newcommand{\sslash}{/\!\!/}

\usepackage{amsthm}
\usepackage{thmtools}
\usepackage[noabbrev]{cleveref}

\newtheorem{theorem}{Theorem}[section]
\newtheorem{proposition}[theorem]{Proposition}
\newtheorem{lemma}[theorem]{Lemma}
\newtheorem{corollary}[theorem]{Corollary}
\newtheorem{conjecture}[theorem]{Conjecture}

\theoremstyle{definition}
\newtheorem{definition}[theorem]{Definition}

\newtheorem{remark}[theorem]{Remark}

\begin{document}

\begin{abstract}
    We prove that the asymptotics of the Hermitian-Yang-Mills flow on a slope semistable holomorphic vector bundle over a compact K\"ahler manifold are determined algebro-geometrically, via the iterated filtration defined by Haiden-Katzarkov-Kontsevich-Pandit. This proves a conjecture of Haiden-Katzarkov-Kontsevich-Pandit in this setting. Moreover, we prove a non-linear analogue, relating the asymptotics of the Calabi flow near a cscK manifold to the iterated balancing filtration of the deformation space. In both settings, we reduce the infinite-dimensional flow to a finite-dimensional flow, following the foundational work of Chen-Sun. In finite dimensions, we prove that the asymptotics of the moment map flow are determined by the iterated balancing filtration, proving a conjecture of Ib\'a\~nez N\'u\~nez.
\end{abstract}

\maketitle

\section{Introduction}

In complex geometry, a guiding principle is that geometric PDEs should arise as moment maps. Two prominent examples are the \emph{constant scalar curvature K\"ahler (cscK)} equation, where the moment map interpretation is due to Donaldson \cite{donaldsonRemarksGaugeTheory1997} and Fujiki \cite{fujikiModuliSpacePolarized1992}, and the \emph{Hermitian-Yang-Mills} equation, where the moment map interpretation is due to Atiyah-Bott \cite{atiyahYangMillsEquationsRiemann1983} and Donaldson \cite{donaldsonSelfdualYangMillsConnections1985}. The moment map property of these equations motivates the link between solutions to the PDEs and algebro-geometric stability conditions, namely the Yau-Tian-Donaldson conjecture for the cscK equation \cite{yauOpenProblemsGeometry1993,tianKahlerEinsteinMetricsPositive1997,donaldsonScalarCurvatureStability2002}, and the Hitchin-Kobayashi correspondence, due to Donaldson-Uhlenbeck-Yau for the Hermitian-Yang-Mills equation \cite{donaldsonSelfdualYangMillsConnections1985,donaldsonInfiniteDeterminantsStable1987,uhlenbeckExistenceHermitianYangMillsConnections1986}. In this paper, we relate the asymptotics of the corresponding gradient flow to algebro-geometric stability, providing a parabolic analogue of the above results.

As part of their programme on \emph{categorical K\"ahler geometry} \cite{haidenCategoricalKahlerGeometry}, Haiden-Katzarkov-Kontsevich-Pandit conjectured the following. Given a triangulated category \(\mcC\), for each object \(E\) of \(\mcC\), we have a space \(\Met(E)\) of \emph{metrics} on \(E\), and a convex function \(S \colon \Met(E) \to \R\). Objects \(E\) such that \(S\) has a critical point correspond to \emph{polystable} objects, with respect to a Bridgeland stability condition \cite{bridgelandStabilityConditionsTriangulated2007}. Moreover, the flow under \(-\grad S\) will decompose \(E\) into its polystable components. However, the Harder-Narasimhan filtration only decomposes \(E\) into \emph{semistable} constituents. Thus, a main part of their programme is to understand the decomposition induced by the flow of a semistable object into its polystable components.

Given an artinian abelian category \(\mcA\), and a homomorphism \(X \colon K_0(\mcA) \to \R\) which is positive on each class of a non-zero object, for each object \(E \in \mcA\), Haiden-Katzarkov-Kontsevich-Pandit \cite{haidenSemistabilityModularLattices2023} define a filtration of \(E\), labelled by \(\R^\infty = \R^{\oplus \N}\) with the lexicographic order, which we will refer to as the \emph{iterated HKKP filtration} in this paper. They conjecture that the asymptotics of the flow are determined by the iterated HKKP filtration, and prove this conjecture for quiver representations \cite{haidenSemistabilityModularLattices2023} and for holomorphic vector bundles over compact Riemann surfaces \cite{haidenIteratedLogarithmsGradient2018}.

We develop a new technique to prove their conjecture in various cases. Our first result is to prove their conjecture for holomorphic vector bundles over smooth projective varieties of arbitrary dimension, assuming the graded object of the Jordan-H\"older filtration is locally free.

Let \((X, \omega)\) be a compact K\"ahler manifold, \(\mcE\) a slope semistable holomorphic vector bundle over \(X\). Associated to \(\mcE\) is a Jordan-H\"older filtration of coherent subsheaves
\[0 = \mcE_0 \subset \mcE_1 \subset \dots \subset \mcE_n = \mcE,\]
such that the \(\mcE_i/\mcE_{i-1}\) are slope stable. The graded object of the Jordan-H\"older filtration is
\[\gr(\mcE) = \bigoplus_i \mcE_i/\mcE_{i-1},\]
which we will assume to be locally free. Thus, it is a slope polystable vector bundle, and hence admits a Hermite-Einstein metric \(h_{\gr(\mcE)}\). Let \(K = \Aut(\gr(\mcE), h_{\gr(\mcE)})\) denote the group of unitary automorphisms of the holomorphic vector bundle \(\gr(\mcE)\).

Given a hermitian metric \(h\), we let \(F_h\) denote the curvature of the Chern connection with respect to \(h\). The \emph{Hermitian-Yang-Mills flow} is given by
\[h^{-1}\pdv{h}{t} = -2\left(i\Lambda_\omega F_h - c\id_\mcE\right),\]
where \(c\) is a cohomological constant. It is a result of Donaldson \cite{donaldsonSelfdualYangMillsConnections1985} that the flow exists for all time. The Hermitian-Yang-Mills flow plays a central role in connecting Hermite-Einstein metrics to algebraic geometry, and also in the geometry of moduli spaces.

To state the result, we recall that the matrix exponential \(\exp\) defines a diffeomorphism from the space of hermitian endomorphisms of \(\mcE\) to the space of hermitian metrics on \(\mcE\), and we denote its inverse by \(\log\).

\begin{theorem}
    [\Cref{thm:ym}]
    \label{thm:ym-intro}

    Suppose \(h(t)\) is a solution to the Hermitian-Yang-Mills flow. Then there exists \(v_1, \dots, v_r \in \Lie(K)\), such that
    \[-i\log(h(t)) = \log(t)v_1 + \log\log(t) v_2 + \dots + \log \cdots \log(t)v_r + O(1),\]
    where \(O(1)\) denotes a bounded term. Moreover, the \(v_i\) pairwise commute, and are given by the iterated HKKP filtration of the lattice of subbundles of \(\mcE\).
\end{theorem}

We expect the result to also hold when \(\gr(\mcE)\) is not locally free, where in this setting, one should consider the lattice of subobjects in the quotient category \(\Coh(X) / \Coh^{\ge 2}(X)\), with \(\Coh^{\ge 2}(X)\) the subcategory of coherent sheaves supported in codimension at least 2. We refer to \Cref{rmk:ym-lattice} for more details.

Our technique to prove \Cref{thm:ym-intro} is general, and applies to geometric flows not arising from the linear setting of abelian categories. As a non-linear analogue of the above result, we prove a similar result for the Calabi flow, which is the natural geometric flow involved in the theory of cscK metrics. Let \((X_0, L_0)\) be a smooth projective variety, \(\omega_0 \in c_1(L_0)\) a cscK metric. Suppose \((X, L)\) is a sufficiently small deformation of \((X_0, L_0)\), which we can view as a K-semistable polarised variety, with smooth and cscK K-polystable degeneration. 

The \emph{Calabi flow} on \((X, L)\) is given by
\[\dv{\phi}{t} = S(\phi(t)) - \underline S,\]
where \(\underline S\) is a cohomological constant, and \(S(\phi)\) is the hermitian scalar curvature functional. It is a result of Chen-Sun \cite{chenCalabiFlowGeodesic2014} that the flow exists for all time in this setting. We let \(K = \Ham(X_0, \omega_0) \cap \Aut(X_0)\) denote the group of Hamiltonian isometries of \((X_0, \omega_0)\), and let \(G\) denote the complexification of \(K\).

\begin{theorem}
    [\Cref{thm:calabi-flow}]
    \label{thm:calabi-flow-intro}

    Suppose \(\phi(t)\) is a solution to the Calabi flow. Then there exists \(v_1, \dots, v_k \in \Lie(K)\), such that
    \[\phi(t) = \exp(i\log(t)v_1 + \dots + i\log\cdots\log(t)v_k) + O(1),\]
    where \(v_1, \dots, v_k\) are commuting elements of \(\mfk\), and \(\exp \colon i\mfk \to G/K\) is the exponential map.
\end{theorem}

Up to scaling, the elements \(v_i \in \Lie(K)\) are rational, and hence correspond to (rational) one-parameter subgroups of \(G = \Aut(X_0, L_0)\). Thus, each \(v_i\) corresponds to a test configuration of \((X, L)\), with central fibre \((X_0, L_0)\), see \cite[Proof of Theorem 2]{szekelyhidiKahlerRicciFlowKpolystability2010} and \cite{chenCalabiFlowGeodesic2014}. It would be an interesting question to construct the \(v_i\) directly, in terms of K-stability.

Motivated by the conjectures of Haiden-Katzarkov-Kontsevich-Pandit, and Kirwan's refinements of the Morse stratification in geometric invariant theory \cite{kirwanRefinementsMorseStratification2005}, Ib\'a\~nez N\'u\~nez defines the \emph{iterated balancing filtration} of a good moduli stack \cite{ibaneznunezRefinedHarderNarasimhanFiltrations2024}, which recovers the HKKP filtration for moduli stacks of objects in an abelian category, and Kirwan's refinement of the Morse stratification for GIT quotients. Ib\'a\~nez N\'u\~nez conjectures that in the setting of a GIT quotient, the asymptotics of the moment map flow can be computed from the iterated balancing filtration \cite[Conjecture 1.7.1]{ibaneznunezRefinedHarderNarasimhanFiltrations2024}.

Let \((X, L)\) be a smooth projective variety, \(G\) be a complex reductive Lie group, acting on \(X\) by biholomorphisms, with linearisation \(L\). We will fix a K\"ahler metric \(\omega \in c_1(L)\). Fix a maximal compact subgroup \(K\) of \(G\), and an \(\Ad\)-invariant inner product on \(\mfk\). The \(K\) action on \((X, \omega)\) is Hamiltonian, with moment map \(\mu \colon X \to \mfk\). This is the setting of geometric invariant theory.

Associated to each \(x \in X\), we have a geodesically convex function \(\Phi_x \colon G/K \to \R\), called the \emph{Kempf-Ness function}. Note that it is often referred to as the \emph{log-norm functional}. The GIT stability of \(x\) is encoded in the asymptotic properties of \(\Phi_x\).

We let \(\gamma_x(t)\) denote the solution to the gradient flow
\[\begin{cases}
    \dv{\gamma_x}{t} &= -\grad \Phi_x(\gamma_x(t)) \\
    \gamma_x(0) &= [1] \in G/K.
\end{cases}\]
This flow plays a prominent role in understanding the geometry of GIT quotients.

The exponential map \(\exp \colon i\Lie(K) \to G/K\) is a diffeomorphism, and we denote its inverse by \(\log \colon G/K \to i\Lie(K)\).

\begin{theorem}
    [\Cref{thm:main,thm:iterated-balancing-filtration}]
    \label{thm:main-intro}
    Suppose \(x \in X\) is a semistable point. Then there exists \(v_1, \dots, v_r \in \Lie(K)\), such that
    \[-i\log(\gamma_x(t)) = \log(t)v_1 + \log\log(t)v_2 + \cdots + \log\cdots\log(t)v_r + O(1).\]
    Moreover, the \(v_i\) pairwise commute, and are given by the iterated balancing filtration of Ib\'a\~nez N\'u\~nez.
\end{theorem}

In this setting, the iterated balancing filtration agrees with the refinement of the Morse stratification by Kirwan \cite[Section 3.6.2]{ibaneznunezRefinedHarderNarasimhanFiltrations2024}. Thus, our work also connects the asymptotics of the moment map flow to the refined Morse stratification. There is a closely related result in the linear setting.

\begin{theorem}
    [{\Cref{thm:vector-space}}]
    \label{thm:vector-space-intro}
    Let \((V, J_0, g_0)\) be a unitary representation of a compact connected Lie group \(K\), which extends to a representation of the complexification \(G\). Let \(\Omega_0\) denote the induced K\"ahler form on \(V\). Let \(\mu \colon V \to \mfk\) denote the associated moment map. Let \(x \in V \setminus \{0\}\) be such that \(0 \in \overline{G \cdot x}\). Let \(\gamma_x(t)\) denote the flow line for \(-\grad \Phi_x\) starting at \([1]\). Then there exist elements \(v_1, \dots, v_k\) of \(\mfk\), such that
    \[-i\log(\gamma_x(t)) = \log(t)v_1 + \log\log(t)v_2 + \log\cdots\log(t)v_k + O(1).\]
    Moreover, the \(v_i\) pairwise commute, and are given by the iterated balancing filtration of Ib\'a\~nez N\'u\~nez.
\end{theorem}

When \(X = \Rep(Q, d)\) is the space of representations of a quiver \(Q\) with dimension vector \(d\) and \(G = G(d)\), Ib\'a\~nez N\'u\~nez proves that the iterated balancing filtration agrees with the HKKP filtration \cite[Theorem 1.6.1]{ibaneznunezRefinedHarderNarasimhanFiltrations2024}. Thus, we recover the main result of Haiden-Katzarkov-Kontsevich-Pandit \cite[Theorem 5.11]{haidenSemistabilityModularLattices2023}, who proved this result for quiver representations, using different techniques.

Finally, we have an analogous result in the unstable case.

\begin{theorem}
    [\Cref{thm:unstable}]
    \label{thm:unstable-intro}
    Suppose \(x \in X\) is unstable. Then there exists \(v_0, v_1, \dots, v_r \in \Lie(K)\), such that
    \[-i\log(\gamma_x(t)) = tv_0 + \log(t)v_1 + \log\log(t)v_2 + \cdots + \log\cdots\log(t)v_r + O(1).\]
    Moreover, the \(v_i\) pairwise commute, and \(v_1, \dots, v_r\) are given by the iterated balancing filtration of Ib\'a\~nez N\'u\~nez.
\end{theorem}

These results prove a conjecture of Ib\'a\~nez N\'u\~nez \cite[Conjecture 1.7.1]{ibaneznunezRefinedHarderNarasimhanFiltrations2024}. The goal of Ib\'a\~nez N\'u\~nez's iterated balancing filtration is to describe the complete asymptotics of the moment map flow, and \Cref{thm:main-intro,thm:vector-space-intro,thm:unstable-intro} demonstrate that this is the case.

\begin{remark}
    In \cite[Conjecture 1.7.1]{ibaneznunezRefinedHarderNarasimhanFiltrations2024}, Ib\'a\~nez N\'u\~nez allows the linearisation for the \(G\) action on \(V\) to vary, which corresponds to shifting the moment map by a central element. In this setting, it is a result of Harada-Wilkin \cite{haradaMorseTheoryMoment2011} that the moment map flow converges. As the proof of \Cref{thm:main-intro} relies on reducing to a local neighbourhood, the same techniques prove \cite[Conjecture 1.7.1]{ibaneznunezRefinedHarderNarasimhanFiltrations2024} in full generality.
\end{remark}

The proofs of the infinite-dimensional results rely on reduction to a finite-dimensional slice, analogous to the work of Chen-Sun \cite{chenCalabiFlowGeodesic2014}, and then applying \Cref{thm:vector-space-intro}. Thus, the \(v_i\) in the infinite-dimensional settings are given by the iterated balancing filtration of the finite-dimensional deformation space. In \Cref{prop:kuranishi-slice-quiver}, we identify the slice for the Hermitian-Yang-Mills flow with a quiver representation, and relate the lattice of subrepresentations to the lattice of subbundles, and hence show that the \(v_i\) can be computed from the HKKP filtration of the lattice of subbundles.

\subsection{Prior work}

The construction of stratifications of the unstable locus has a long history. Using Kempf's optimal destabilising one-parameter subgroup \cite{kempfInstabilityInvariantTheory1978}, Hesselink defines a stratification of the unstable locus of a GIT quotient \cite{hesselinkUniformInstabilityReductive1978}. In \cite{kirwanCohomologyQuotientsSymplectic1984,nessStratificationNullCone1984} Kirwan and Ness prove that Hesselink's stratification coincides with the Morse stratification, induced by the gradient flow of the norm squared of the moment map. 

In \cite{atiyahYangMillsEquationsRiemann1983}, Atiyah-Bott outlined a programme to use the Yang-Mills functional to define a stratification of the space of unitary connections on a vector bundle over a compact Riemann surface, where the strata are labelled by Harder-Narasimhan type. The programme was completed by Daskalopoulos \cite{daskalopoulosTopologySpaceStable1992} and R{\aa}de \cite{radeYangMillsHeatEquation1992}, who studied the analytic properties of the Yang-Mills flow.

In both the finite- and infinite-dimensional settings, the gradient flow of the norm squared of the moment map defines a deformation retraction, which can be used to compute the cohomology of the quotient space (respectively, moduli space) \cite{kirwanCohomologyQuotientsSymplectic1984,atiyahYangMillsEquationsRiemann1983,daskalopoulosTopologySpaceStable1992}. Motivated by this, using the techniques from \cite{kirwanPartialDesingularisationsQuotients1985}, Kirwan defines a refinement of the Morse stratification, by iteratively stratifying the semistable locus via blow-ups \cite{kirwanRefinementsMorseStratification2005}. Using the same ideas, Kirwan defines refinements of the Morse stratification for the Yang-Mills functional \cite{kirwanModuliSpacesBundles2004}. In particular, for a semistable bundle over a compact Riemann surface, Kirwan's balanced \(\delta\)-filtration of maximal triviality corresponds to the (not iterated) HKKP filtration of the lattice of subbundles of the same slope \cite[Section 3.6.4]{ibaneznunezRefinedHarderNarasimhanFiltrations2024}.

Convergence of the Hermitian-Yang-Mills flow was established by Daskalopoulos \cite{daskalopoulosTopologySpaceStable1992} and R{\aa}de \cite{radeYangMillsHeatEquation1992} for Riemann surfaces, Daskalopoulos-Wentworth \cite{daskalopoulosConvergencePropertiesYangMills2004} for K\"ahler surfaces, and Jacob \cite{jacobLimitYangMillsFlow2015,jacobYangMillsFlowAtiyahBott2016} and Sibley \cite{sibleyAsymptoticsYangMillsFlow2015} for general K\"ahler manifolds. In particular, for a slope semistable vector bundle \(\mcE\) with \(\gr(\mcE)\) locally free, it was proven that the Yang-Mills flow for connections converges to \(\gr(\mcE)\). The iterated HKKP filtration of \(\mcE\) is
\[0 = \mcS_0 \subset \mcS_1 \subset \cdots \subset \mcS_r = \mcE,\]
with \(\mcS_i/\mcS_{i-1}\) \emph{polystable} \cite{haidenIteratedLogarithmsGradient2018}. Thus,
\[\gr_{\mathrm{HKKP}}(\mcE) = \bigoplus_i \frac{\mcS_i}{\mcS_{i-1}}\]
is the same as \(\gr(\mcE)\), and one can choose the Jordan-H\"older filtration to be a refinement of the iterated HKKP filtration. Thus, we can view \Cref{thm:ym-intro} as a refinement of the convergence results of the Hermitian-Yang-Mills flow.

The convergence results for quiver representations were established by Harada-Wilkin \cite{haradaMorseTheoryMoment2011}. In \cite{hoskinsStratificationsAssociatedReductive2014}, Hoskins shows that the Morse theoretic stratification agrees with Hesselink's stratification in this setting.

Finally, in \cite{chenCalabiFlowGeodesic2014}, Chen-Sun study the asymptotics of the gradient flow of the norm squared of the moment map, using the {\L}ojasiewicz inequality. Moreover, they apply the result to the Calabi flow, by reducing to a finite-dimensional slice, proving that the Calabi flow is asymptotic to a geodesic ray. In turn, this result is used to prove the uniqueness of cscK degenerations. The finite-dimensional results from \cite{chenCalabiFlowGeodesic2014} can also be found in \cite{georgoulasMomentweightInequalityHilbertMumford2021}. Many of the analytic results used in this paper, and the reduction to a finite-dimensional slice, are based on the techniques developed by Chen-Sun.

\begin{remark}
    Many of the results used in this paper from Ib\'a\~nez N\'u\~nez's DPhil thesis \cite{ibaneznunezRefinedHarderNarasimhanFiltrations2024} can be found in the preprint \cite{ibaneznunezRefinedHarderNarasimhanFiltrations2023}, which contains the first half of \cite{ibaneznunezRefinedHarderNarasimhanFiltrations2024}. For consistency, throughout all references are to \cite{ibaneznunezRefinedHarderNarasimhanFiltrations2024}.
\end{remark}

\subsection{Outlook}

Many geometric PDEs can be written as moment maps, and we expect the general strategy of reducing to a finite-dimensional slice should apply, proving analogous results for the asymptotics of the flow. We briefly discuss some examples, though the conjectures as stated may be too na\"ive.

The first example is the special Lagrangian equation. It is explained by Thomas \cite{thomasMomentMapsMonodromy2001} that it arises as a moment map, and by Thomas-Yau \cite{thomasSpecialLagrangiansStable2002} that for a specific choice of inner product on the Lie algebra, the gradient flow is Lagrangian mean curvature flow. We also refer to \cite[Section 38.4]{horiMirrorSymmetry2003}. Motivated by mirror symmetry, Solomon constructs a non-positively curved space, along with a geodesically convex functional, whose critical points are the special Lagrangian submanifolds \cite{solomonCalabiHomomorphismLagrangian2013,solomonCurvatureSpacePositive2014}. The following is a variant of a conjecture by Haiden-Katzarkov-Kontsevich-Pandit.

\begin{conjecture}
    \label{conj:lagrangian}
    Let \((X, \omega, \Omega)\) be a compact Calabi-Yau manifold, \(L_1, \dots, L_r\) special Lagrangian submanifolds in \(X\), with the same phase, and intersecting transversely. Let \(L = L_1 \# \cdots \# L_r\) denote the result of performing Lagrangian surgery near the intersection points. Then the asymptotics of the Lagrangian mean curvature flow with initial condition \(L\) are given by sums of iterated logarithms.
\end{conjecture}

Our conjecture is closely related to the Thomas-Yau conjecture, for which we refer to the articles of Joyce \cite{joyceConjecturesBridgelandStability2015} and Li \cite{liThomasYauConjectureHolomorphic2025}, and references therein. The geometric setting in \Cref{conj:lagrangian} is intended to be the semistable case, under an appropriate stability condition. We note that when \(r = 1\), Lagrangian mean curvature flow near a special Lagrangian converges exponentially fast \cite{liConvergenceLagrangianMean2012}. An example of the above set-up with \(r = 2\) was constructed by Su-Tsai-Wood \cite{suInfinitetimeSingularitiesLagrangian2024}, where they prove estimates on the asymptotic behaviour.

In \cite[Conjecture 5.1]{haidenIteratedLogarithmsGradient2018}, Haiden-Katzarkov-Kontsevich-Pandit make a conjecture about a modified curve shortening flow, motivated by stability in partially wrapped Fukaya categories, as an A-side analogue of their result on the Yang-Mills flow on Riemann surfaces.

On the B-side, for the deformed Hermitian-Yang-Mills equation, Collins-Yau \cite{collinsMomentMapsNonlinear2021} show that it can be interpreted as a moment map, as a mirror of Thomas and Solomon's results on the A-side. We also refer to the survey by Collins-Shi \cite{collinsStabilityDeformedHermitianYangMills2022} and references therein. In \cite{takahashiTanconcavityPropertyLagrangian2020}, Takahashi introduces the \emph{tangent Lagrangian Phase flow}, which can be interpreted as the gradient flow of the Kempf-Ness functional. We expect an analogue of \Cref{conj:lagrangian} should hold in this setting. Many other flows have been introduced in this setting, such as the \emph{line bundle mean curvature flow} by Jacob-Yau \cite{jacobSpecialLagrangianType2017}, and the \emph{cotangent flow} by Fu-Yau-Zhang \cite{fuNewFlowSolving2024}. It would be an interesting question to understand the asymptotics of these flows, and also in the more general setting of the \(Z\)-critical equation, introduced by Dervan-McCarthy-Sektnan \cite{dervanZcriticalConnectionsBridgeland2024}, which is related to Bridgeland stability conditions on \(D^b\Coh(X)\).

Finally, the Calabi flow is a fourth-order fully non-linear PDE, which poses many analytic difficulties. The K\"ahler-Ricci flow is a second-order PDE, and we conjecture that the asymptotics of the K\"ahler-Ricci flow can also be computed from the iterated balancing filtration.

\begin{conjecture}
    Let \(X\) be a K-semistable Fano manifold. Then an analogous result holds for the K\"ahler-Ricci flow starting from \(\omega\).
\end{conjecture}

One key obstruction is that the K\"ahler-Ricci flow is not given by the gradient flow of the norm squared of a moment map, and it may be the case that the logarithm is replaced by a different function. However, by a result of Donaldson \cite{donaldsonDingFunctionalBerndtsson2017}, it is the gradient flow of a convex function of a moment map. As a finite-dimensional analogue, we conjecture that a similar result to \Cref{thm:main-intro} should hold for the gradient flow of a convex function of the moment map, as studied by Lee-Sturm-Wang \cite{leeMomentMapConvex2025}.

\subsection{Organisation}

First, in \Cref{sec:lojasiewicz}, we recall the {\L}ojasiewicz inequality for real-analytic functions, and its applications to gradient flows. Next, in \Cref{sec:moment-maps}, we recall the moment map geometry from \cite{georgoulasMomentweightInequalityHilbertMumford2021} required in this paper. In \Cref{sec:normal-forms}, we study the reduction to a standard form, following \cite{chenCalabiFlowGeodesic2014}. In \Cref{sec:finite-dimensional}, we prove \Cref{thm:main-intro,thm:vector-space-intro,thm:unstable-intro}, and in \Cref{sec:iterated-balancing-filtration}, we relate the \(v_i\) to the iterated balancing filtration of Ib\'a\~nez N\'u\~nez. Finally, in \Cref{sec:quiver-representations}, \Cref{sec:ym} and \Cref{sec:calabi-flow}, we apply our main result to quiver representations, Hermitian-Yang-Mills flow and the Calabi flow respectively.

\subsection{Acknowledgements}

I would like to thank my PhD supervisor Ruadha\'i Dervan for suggesting this problem, constant encouragement and many helpful discussions. I would also like to thank Fabian Haiden, Jacopo Stoppa and Peter Topping for helpful discussions on this work, and Maxwell Stolarski for a helpful discussion on Lagrangian mean curvature flow. I would like to thank Andr\'es Ib\'a\~nez N\'u\~nez for discussions on \cite{ibaneznunezRefinedHarderNarasimhanFiltrations2024}, and extensive comments and corrections on an earlier draft of this paper.

I was supported by a PhD studentship associated to Ruadha\'i Dervan's Royal Society University Research Fellowship (URF{\textbackslash}R1{\textbackslash}201041).

\subsection{Notation}

Given two functions \(f, g \colon \R_{>0} \to \R\), we write \(f \lesssim g\) if there exists a constant \(C > 0\) such that \(f(t) \le Cg(t)\) for all \(t\) sufficiently large. We write \(f \asymp g\) if both \(f \lesssim g\) and \(g \lesssim f\). For a parameter \(t\), we write \(O_{t \to \infty}(1)\) to denote a term which is bounded as \(t \to \infty\).

\section{{\L}ojasiewicz inequality}

\label{sec:lojasiewicz}

One important tool in the study of gradient flows is the {\L}ojasiewicz inequality for real-analytic functions \cite{lojasiewiczProprieteTopologiqueSousensembles1963}, which we now recall. The following statements are from Chen-Sun \cite{chenCalabiFlowGeodesic2014}.

\begin{theorem}
    [{\L}ojasiewicz inequality] Suppose \(f\) is a real-analytic function defined on a neighbourhood \(U\) of \(0 \in \R^n\). Suppose \(f(0) = 0\) and \(\grad f(0) = 0\). Then there exists \(C > 0\) and \(\alpha \in [1/2, 1)\), and an open neighbourhood \(V\) of \(0\) contained in \(U\), such that for all \(x \in V\),
    \[\abs{\grad f(x)} \ge C\abs{f(x)}^\alpha.\]
\end{theorem}

One can apply the {\L}ojasiewicz inequality to study the asymptotics of gradient flows.

\begin{corollary}
    \label{cor:lojasiewicz}
    Suppose \(f\) is a nonnegative real-analytic function defined in a neighbourhood \(U\) of \(0 \in \R^n\), with \(f(0) = 0\), \(\grad f(0) = 0\) and {\L}ojasiewicz exponent \(\alpha \in (1/2, 1)\). Then there exists a neighbourhood \(V \subseteq U\) of \(0\), such that for all \(x_0 \in V\), the gradient flow
    \[\begin{cases}
        \dv{x}{t} &= -\grad f(x(t)) \\
        x(0) &= x_0
    \end{cases}\]
    converges to a limit \(x_\infty \in U\), with \(f(x_\infty) = 0\). Moreover, we have the following estimates.
    \begin{enumerate}[(i)]
        \item \(f(x(t)) \lesssim t^{-1/(2\alpha - 1)}\),
        \item \(d(x(t), x_\infty) \lesssim t^{-(1 - \alpha)/(2\alpha - 1)}.\)
    \end{enumerate}
\end{corollary}

\section{Moment maps}

\label{sec:moment-maps}

In this section, we recall the moment map geometry required for this paper. All of the results can be found in \cite{georgoulasMomentweightInequalityHilbertMumford2021}, which builds on the results of \cite{chenCalabiFlowGeodesic2014} and more classical work \cite{kirwanCohomologyQuotientsSymplectic1984,nessStratificationNullCone1984,mumfordGeometricInvariantTheory1994}. We follow the same conventions as \cite{georgoulasMomentweightInequalityHilbertMumford2021}.

Throughout, suppose \(X \subseteq \P^n\) is a smooth projective variety, \(G\) a reductive complex Lie group acting linearly on \(\P^n\), preserving \(X\). Let \(\omega\) denote the restriction of the Fubini-Study form to \(X\). Suppose \(K\) is a maximal compact subgroup of \(G\), acting by isometries on \(\P^n\). Finally, fix an \(\Ad\)-invariant inner product on \(\mfk = \Lie(K)\), so that we may identify \(\mfk \cong \mfk^*\). We assume that the \(G\) action on \(X\) is on the \emph{left}, and the space \(G/K = \{gK \mid g \in G\}\) denotes the space of \emph{left} cosets.

\begin{remark}
    Most of the results in this section hold for general compact K\"ahler manifolds, though we will not need them in full generality. Other than in \Cref{sec:normal-forms}, we will only apply the results in the case \(X = \P^n\).
    
    Moreover, the compactness assumption is only needed to show various limits exist. We refer to Harada-Wilkin \cite{haradaMorseTheoryMoment2011}, and also \Cref{sec:normal-forms}, for a discussion on moment map flow on vector spaces.
\end{remark}

For \(\xi \in \mfk\), we define
\[v_\xi(x) = \dv{t}\bigg\vert_{t=0}\exp(t\xi) \cdot x\]
for the corresponding vector field on \(X\). We write
\[L_x \colon \mfk \to T_xX\]
for the infinitesimal action of \(\mfk\) at \(x\). This extends naturally to a map \(L_x^c \colon \mfg \to T_xX\) by complex linearity.

\begin{definition}
    A \emph{moment map} for the \(K\) action on \((X, \omega)\) is a \(K\)-equivariant map \(\mu \colon X \to \mfk\), such that for all \(\xi \in \mfk\), we have
    \[\dd\langle \mu, \xi \rangle = \iota_{v_\xi}\omega.\]
\end{definition}

As we have assumed that \(X\) is a projective variety and that the action is linear, a moment map always exists, and is given by the formula
\[\langle \mu(x), \xi \rangle = \frac{\langle v, i \xi \cdot v\rangle}{\langle v, v \rangle},\]
where \(v \in \C^{n+1}\) is a lift of \(x\).

\subsection{Norm squared of the moment map}

We let \(f \colon X \to \R\) be the function
\[f(x) = \frac{1}{2}\abs{\mu(x)}^2.\]

\begin{lemma}
    [{\cite[Lemma 3.1]{georgoulasMomentweightInequalityHilbertMumford2021}}]
    The gradient of \(f\) is given by
    \[\grad f(x) = JL_x\mu(x).\]
\end{lemma}

For \(x_0 \in X\), we let \(x(t)\) denote the solution to the ODE
\[\begin{cases}
    \dv{x}{t} &= -JL_x\mu(x) \\
    x(0) &= x_0.
\end{cases}\]

\begin{lemma}
    [{\cite[Lemma 3.2]{georgoulasMomentweightInequalityHilbertMumford2021}}]
    \label{lem:group-flow}
    Let \(x_0 \in X\). Let \(g_{x_0}(t)\) denote the unique solution to the ODE
    \[\begin{cases}
        g_{x_0}^{-1}\dv{g_{x_0}}{t} &= i\mu(x(t)) \\
        g_{x_0}(0) &= 1.
    \end{cases}\]
    Then
    \[x(t) = g_{x_0}(t)^{-1}\cdot x_0.\]
\end{lemma}
As a corollary, this implies that the flow \(x(t)\) stays in the same \(G\)-orbit.

\subsection{Kempf-Ness function}

Associated to each \(x \in X\), we can define a function \(\Phi_x \colon G/K \to \R\).

\begin{proposition}
    [{\cite[Theorem 4.1]{georgoulasMomentweightInequalityHilbertMumford2021}}]
    Fix an element \(x \in X\). 
    \begin{enumerate}[(i)]
        \item There exists a unique function \(\Phi_x \colon G \to \R\), such that
        \[\dd\Phi_x(g)\zeta = -\langle \mu(g^{-1}x), \Im(g^{-1}\zeta)\rangle\]
        for all \(\zeta \in T_gG\), with \(\Phi_x(u) = 0\) for all \(u \in K\).
        \item Define a map \(a_x \colon G \to X\) by \(a_x(g) = g^{-1}x\). Then for all \(g \in G\),
        \[\dd a_x(g)\left(\grad\Phi_x(g)\right) = \grad f(a_x(g)).\]
    \end{enumerate}
\end{proposition}

One corollary of the definition is that the function \(\Phi_x\) is \(K\)-invariant.

\begin{definition}
    The function \(\Phi_x \colon G/K \to \R\) is called the \emph{Kempf-Ness function} based at \(x\).
\end{definition}

\begin{remark}
    Often, \(\Phi_x\) is also referred to as the \emph{log-norm functional}, due to the expression for \(\Phi_x\) in the case of projective space.
\end{remark}

The main properties of the Kempf-Ness function which we need are the following.

\begin{proposition}
    [{Properties of the Kempf-Ness function, \cite[Theorem 4.3]{georgoulasMomentweightInequalityHilbertMumford2021}}]
    \leavevmode
    \begin{enumerate}[(i)]
        \item The function \(\Phi_x\) is convex along geodesics.
        \item Let \(\pi \colon G \to G/K\) denote the quotient map. Let \(g(t)\) be a curve in \(G\), and let \(\gamma(t) = \pi(g(t))\). Then \(\gamma\) is a flow line for \(-\grad\Phi_x\) if and only if
        \[\Im\left(g^{-1}\dv{g}{t}\right) = \mu(g^{-1}x).\]
    \end{enumerate}
\end{proposition}

An immediate consequence of the above is the following.

\begin{corollary}
    For \(x \in X\), let \(g_x(t)\) be the flow in \(G\) defined in \Cref{lem:group-flow}. Let \(\gamma_x(t) = [g_x(t)] \in G/K\). Then \(\gamma_x\) is a flow line for \(-\grad \Phi_x\).
\end{corollary}

\subsection{Stability}

Using the moment map, we can define a notion of stability.

\begin{definition}
    A point \(x \in X\) is called
    \begin{enumerate}[(i)]
        \item \emph{\(\mu\)-unstable} if \(\overline{G \cdot x} \cap \mu^{-1}(0) = \emptyset\).
        \item \emph{\(\mu\)-semistable} if \(\overline{G \cdot x} \cap \mu^{-1}(0) \neq \emptyset\).
        \item \emph{\(\mu\)-polystable} if \(G \cdot x \cap \mu^{-1}(0) \neq \emptyset\).
        \item \emph{\(\mu\)-stable} if \(x\) is \(\mu\)-polystable, and the stabiliser \(G_x\) of \(x\) is discrete.
    \end{enumerate}
\end{definition}

We write \(X^{\ss}, X^{\ps}\) and \(X^{\st}\) for the sets of \(\mu\)-semistable, \(\mu\)-polystable and \(\mu\)-stable points respectively. Moreover, one can define stability in terms of the flow.

\begin{theorem}
    [{\cite[Theorem 7.2]{georgoulasMomentweightInequalityHilbertMumford2021}}]
    Let \(x_0 \in X\), and let \(x(t)\) be the flow line for \(-\grad f\) starting at \(x_0\). Let \(x_\infty = \lim_{t \to \infty} x(t)\). Then
    \begin{enumerate}[(i)]
        \item \(x_0\) is \(\mu\)-semistable if and only if \(\mu(x_\infty) = 0\),
        \item \(x_0\) is \(\mu\)-polystable if and only if \(\mu(x_\infty) = 0\) and \(x_\infty \in G \cdot x_0\),
        \item \(x_0\) is \(\mu\)-stable if and only if \(\mu(x_\infty) = 0\), and the stabiliser \(G_{x_\infty}\) of \(x_\infty\) is discrete.
    \end{enumerate}
\end{theorem}

The proof of (i) follows from the following result.

\begin{theorem}
    [{Moment limit, \cite[Theorem 6.4]{georgoulasMomentweightInequalityHilbertMumford2021}}]
    Let \(x_0 \in X\), and let \(x(t) = \psi^t(x_0)\) be the flow line for \(-\grad f\) starting at \(x_0\). Let \(x_\infty = \lim_{t \to \infty} x(t)\). Then
    \[\abs{\mu(x_\infty)} = \inf_{g \in G}\abs{\mu(g \cdot x_0)}.\]
\end{theorem}

When \(X = \P^n\), we have an algebro-geometric characterisation of stability.

\begin{theorem}
    [Kempf-Ness {\cite{kempfLengthVectorsRepresentation1979}, \cite[Theorem 8.5]{georgoulasMomentweightInequalityHilbertMumford2021}}] Let \(v \in V \setminus \{0\}\), and set \(x = [v] \in \P(V)\).
    \begin{enumerate}[(i)]
        \item \(x\) is \(\mu_\FS\)-unstable if and only if \(0 \in \overline{G \cdot v}\).
        \item \(x\) is \(\mu_\FS\)-semistable if and only if \(0 \notin \overline{G \cdot v}\).
        \item \(x\) is \(\mu_\FS\)-polystable if and only if \(G \cdot v\) is closed.
        \item \(x\) is \(\mu_\FS\)-stable if and only if \(G \cdot v\) is closed and the stabiliser \(G_v\) of \(v\) is discrete.
    \end{enumerate}
    In particular, \(\mu_\FS\)-stability is equivalent to the usual GIT stability.
\end{theorem}

\subsection{Weights}

In geometric invariant theory, the \emph{Hilbert-Mumford criterion} gives a characterisation of stability in terms of one-parameter subgroups. There is a corresponding criterion in the moment map setting, which we now recall.

Let us define
\[\mcT^c = \{g\xi g^{-1} \mid g \in G, \xi \in \mfk \setminus \{0\}\}\]
for the set of \emph{toral generators} in \(\mfg\). Similarly, set
\begin{align*}
    \Lambda &= \{\xi \in \mfk \setminus \{0\} \mid \exp(\xi) = 1\}, \\
    \Lambda^c &= \{\xi \in \mfg \setminus \{0\} \mid \exp(\xi) = 1\}.
\end{align*}
By definition, \(\Lambda \subset \Lambda^c \subset \mcT^c\). Note that we have a one-to-one correspondence between \(\Lambda^c\) and the set of one-parameter subgroups of \(G\). The following result follows from considering \(\C^*\)-actions.

\begin{lemma}
    [{\cite[Lemma 5.4]{georgoulasMomentweightInequalityHilbertMumford2021}}]
    \label{lem:one-param-limit}
    For \(x \in X\), \(\zeta \in \mcT^c\), the limits
    \[x_{\pm} = \lim_{t \to \pm\infty}\exp(it\zeta)x\]
    exist, the convergence is exponential, and \(L^c_{x_\pm}\zeta = 0\).
\end{lemma}

Using this, we can define a notion of \emph{weight}.

\begin{definition}
    The \emph{\(\mu\)-weight} of \((x, \zeta) \in X \times \mcT^c\) is
    \[w_\mu(x, \zeta) = \lim_{t \to \infty}\langle \mu(\exp(it\zeta)x), \Re(\zeta)\rangle.\]
\end{definition}

The \(\mu\)-weight then gives us a characterisation of stability. Analogous results hold for polystability and stability, but we will not need these in this paper.

\begin{theorem}
    [Hilbert-Mumford criterion for semistability, {\cite[Theorem 12.2]{georgoulasMomentweightInequalityHilbertMumford2021}}] 
    \label{thm:hilbert-mumford}
    For \(x_0 \in X\), \(x_0\) is \(\mu\)-semistable if and only if for all \(\zeta \in \mcT^c\), we have \(w_\mu(x_0, \zeta) \ge 0\). In turn, this is equivalent to the condition that for all \(\xi \in \mfk \setminus \{0\}\), we have \(w_\mu(x_0, \xi) \ge 0\).
\end{theorem}

For \(\zeta \in \mcT^c\), we define the parabolic subgroup
\[P(\zeta) = \{g \in G \mid \lim_{t \to \infty}\exp(it\zeta)g\exp(-it\zeta) \text{ exists}\}.\]

One result which we need is the following.

\begin{lemma}
    Let \(\zeta_1, \zeta_2 \in \mcT^c\). Then \(P(\zeta_1) \cap P(\zeta_2)\) contains a maximal torus of \(G\).
\end{lemma}

Finally, we have the following properties of the \(\mu\)-weight.

\begin{lemma}
    [Mumford, {\cite[Theorem 5.3]{georgoulasMomentweightInequalityHilbertMumford2021}}]
    \leavevmode
    \begin{enumerate}[(i)]
        \item For \(x \in X, g \in G, \zeta \in \mcT^c\),
        \[w_\mu(g\cdot x, g \zeta g^{-1}) = w_\mu(x, \zeta),\]
        \item for \(x \in X, \zeta \in \mcT^c, p \in P(\zeta)\),
        \[w_\mu(x, p\zeta p^{-1}) = w_\mu(x, \zeta).\]
    \end{enumerate}
\end{lemma}

\subsection{Kempf existence and Moment-Weight inequality}

A result of Kempf \cite{kempfInstabilityInvariantTheory1978} states that if \(x_0\) is an unstable point, then there exists an \emph{optimal destabilising one-parameter subgroup}. The following result relates the optimal destabilising one-parameter subgroup to the flow of \(-\grad f\).

\begin{theorem}
    [Generalised Kempf existence, {\cite[Theorem 10.4]{georgoulasMomentweightInequalityHilbertMumford2021}}] 
    \label{thm:kempf}
    Let \(x_0 \in X\), and suppose \(x_0\) is \(\mu\)-unstable, so that
    \[m = \inf_{g \in G}\abs{\mu(g \cdot x_0)} > 0.\]
    Let \(x(t)\) denote the flow line for \(-\grad f\), let \(x_\infty = \lim_{t \to \infty} x(t)\). Let \(g_{x_0}(t)\) denote the corresponding flow in \(G\). Now define \(\xi(t) \in \mfk\) and \(u(t) \in K\) by
    \[g_{x_0}(t) = \exp(-i\xi(t))u(t).\]
    Then the limit
    \[\xi_\infty = \lim_{t\to \infty}\frac{\xi(t)}{t}\]
    exists, and satisfies
    \[w_\mu(x_0, \xi_\infty) = -m^2,\quad\abs{\xi_\infty} = m.\]
    Moreover, there exists \(u_\infty \in K\) such that
    \[\xi_\infty = u_\infty \mu(x_\infty)u_\infty^{-1}.\]
\end{theorem}

As a consequence of Kempf's existence result, we have equality in the moment-weight inequality in the unstable case.

\begin{theorem}
    [Moment-Weight, {\cite[Theorem 10.1]{georgoulasMomentweightInequalityHilbertMumford2021}}]
    \label{thm:moment-weight}
    Let \(x_0 \in X\) be \(\mu\)-unstable. Then
    \[\sup_{0 \ne \xi \in \mfk}\frac{-w_\mu(x_0, \xi)}{\abs{\xi}} = \inf_{g \in G}\abs{\mu(g \cdot x_0)}.\]
    Moreover, the supremum on the left hand side is attained at a unique \(\xi_0 \in \mfk\) up to scaling.
\end{theorem}

The direction \(\xi_0\) in \Cref{thm:moment-weight} is the same as the direction \(\xi_\infty\) in \Cref{thm:kempf}, up to scaling. As such, we call \(\xi_0\) the \emph{optimal destabilising direction} for \(x_0\). So far, we have only shown that the optimal destabilising direction \(\xi_0\) exists, but we have not shown that \(\xi_0\) is rational. 

\begin{definition}
    Fix a constant \(\hbar > 0\). An invariant inner product on \(\mfk\) is called \emph{rational} with factor \(\hbar\) if for all \(\xi, \eta \in \Lambda\), with \([\xi, \eta] = 0\), we have that \(\langle \xi, \eta\rangle \in 2\pi\hbar \Z\).
\end{definition}

One example is given by the following.

\begin{lemma}
    \label{lem:lie-rational}
    Suppose \(K\) is a Lie subgroup of \(U(n)\), then the inner product on \(\mfk\) given by
    \[\langle \xi, \eta \rangle = -2\pi\hbar \frac{\tr(\xi\eta)}{4\pi^2}\]
    is rational with factor \(\hbar\).
\end{lemma}

Thus, we have the following result.

\begin{proposition}
    [{\cite[Corollary 10.6]{georgoulasMomentweightInequalityHilbertMumford2021}}]
    \label{prop:rationality}
    Let \(x_0 \in X\) be \(\mu\)-unstable. Suppose the invariant inner product on \(\mfk\) is rational with factor \(\hbar \in \Q_{>0}\). Then there exists a positive integer \(l\) such that \(\sqrt{2\pi l \hbar}\xi_0 \in \Lambda\).
\end{proposition}

\begin{remark}
    We note that when the inner product on \(\mfk\) is rational, then the optimal destabilising direction can be computed algebro-geometrically, and is the same as Kempf's optimal destabilising one-parameter subgroup, up to scaling \cite{kirwanCohomologyQuotientsSymplectic1984,nessStratificationNullCone1984}.
\end{remark}

In what follows, we will not need the precise constant, and just the fact that the torus
\[T = \{\exp(t\xi) \mid t \in \R\}\]
is one-dimensional.

\begin{remark}
    We expect that many of the results in this paper should hold without a rationality assumption, and for general compact K\"ahler manifolds, but we will not pursue this here.
\end{remark}

\subsection{Moment maps relative to a torus}

\label{sec:relative-torus}

In this section, a \emph{torus} means a compact connected abelian Lie group. The results are based on the thesis of Sz\'ekelyhidi \cite{szekelyhidiExtremalMetricsKstability2006}, though we follow the presentation in \cite{georgoulasMomentweightInequalityHilbertMumford2021}. In the subsequent sections, we only need the result when \(T = \{\exp(t\xi) \mid t \in \R\}\) is a one-dimensional compact torus contained in \(K\), and in this setting, many of the results can be found in works of Kirwan \cite{kirwanCohomologyQuotientsSymplectic1984,kirwanPartialDesingularisationsQuotients1985,kirwanRefinementsMorseStratification2005} and Ness \cite{nessStratificationNullCone1984}.

Let \(x \in X\), and let \(T \subseteq G_x\) be a compact torus, with Lie algebra \(\mft \subseteq \mfg\). Let \(G_T\) denote the identity component of the centraliser of \(T\). This has Lie algebra
\[\mfg_T = \{\zeta \in \mfg \mid [\zeta, \tau] = 0 \text{ for all }\tau \in \mft\}.\]
For \(x \in X\), \(\zeta \in \mcT^c \cap \mfg_T \setminus \mft\), the \((\mu, T)\)-weight is defined by
\[w_{\mu, T}(x, \zeta) = \lim_{t\to\infty}\langle\mu(\exp(it\zeta)\cdot x),\Re(\zeta - \pr_T(\zeta))\rangle.\]
Here, \(\pr_T \colon \mfg \to \mft\) is the projection map, defined as follows. On \(\mfg\), we have a \(G\)-invariant pairing
\[\langle \zeta_1, \zeta_2\rangle_\mfg = \langle \Re(\zeta_1), \Re(\zeta_2)\rangle - \langle \Im(\zeta_1), \Im(\zeta_2)\rangle.\]
This is positive definite on \(\mft\), and so there exists a unique linear map \(\pr_T \colon \mfg \to \mft\), such that for all \(\zeta \in \mfg\), \(\tau \in \mft\), we have
\[\langle \zeta - \pr_T(\zeta), \tau\rangle_\mfg = 0.\]

In our setting, \(T\) will in fact be a torus in \(K_x\), and so we now restrict ourselves to this setting. We let \(K_T\) be the identity component of the centraliser of \(T\) in \(K\). This has Lie algebra
\[\mfk_T = \{\xi \in \mfk \mid [\xi, \tau] = 0 \text{ for all }\tau \in \mft\}.\]

In this case, \(K_T/T\) is a compact Lie group, with Lie algebra \(\mfk_T/\mft\). For \(\xi \in \mfk_T\), we write \([\xi]_T = \xi + \mft \in \mfk_T/\mft\). The group \(K_T/T\) acts on the space
\[X_T = \{y \in X \mid T \subseteq K_y\}\]
of fixed points of \(T\).

The following results describe the \(K_T/T\) action on \(X_T\), and how it relates to the original \(K\) action on \(X\). First of all, we have a natural moment map for the \(K_T/T\) action on \(X_T\).

\begin{lemma}
    [{\cite[Lemma 13.9]{georgoulasMomentweightInequalityHilbertMumford2021}}]
    \leavevmode
    \begin{enumerate}[(i)]
        \item The set \(X_T\) is a closed \(K_T\)-invariant complex submanifold of \(X\), and \(\mu(X_T) \subseteq \mfk_T\).
        \item The group \(K_T/T\) acts on \(X_T\) by K\"ahler isometries, with moment map
        \begin{align*}
            \mu_T\colon X_T &\to \mfk_T/\mft \\
            \mu_T(y) &= [\mu(y)]_T.
        \end{align*}
    \end{enumerate}
\end{lemma}

Next, we have the following result about the toral generators.

\begin{lemma}
    [{\cite[Lemma 13.11]{georgoulasMomentweightInequalityHilbertMumford2021}}]
    Let \(T^c\) (resp. \(\mft^c\)) denote the complexification of \(T\) (resp. \(\mft\)). Then the group \(G_T/T^c\) is reductive, with Lie algebra \(\mfg_T/\mft^c\). In fact, \(G_T/T^c\) is the complexification of \(K_T/T\).
    
    Moreover, the set of toral generators in \(\mfg_T/\mft^c\) is given by
    \[\mcT^c_{\mfg_T/\mft^c} = \{[\zeta]_{T^c} \mid \zeta \in \mcT^c \cap \mfg_T \setminus \mft\}.\]
\end{lemma}

Using the above, we can relate the \(\mu\)-weight of \((x, \zeta)\) to the \(\mu_T\)-weight of \((x, [\zeta]_{T^c})\).

\begin{lemma}
    [{\cite[Lemma 13.12]{georgoulasMomentweightInequalityHilbertMumford2021}}]
    \label{lem:relative-torus-weight}
    Let \(x \in X\), and suppose \(T \subseteq K_x\) is a torus. Then
    \begin{enumerate}[(i)]
        \item For \(\zeta \in \mcT^c \cap \mfg_T \setminus \mft\), we have
        \[w_{\mu, T}(x, \zeta) = w_\mu(x, \zeta) - \langle \mu(x), \pr_T(\zeta)\rangle.\]
        \item For \(\zeta \in \mcT^c \cap \mfg_T \setminus \mft\), we have
        \[w_{\mu, T}(x, \zeta) = w_{\mu_T}(x, [\zeta]_{T^c}).\]
        The right hand side denotes the \(\mu_T\) weight of \((x, [\zeta]_{T^c})\) with respect to the \(K_T/T\) action on \(X_T\).
    \end{enumerate}
\end{lemma}

To conclude, we have the following results about the moment map \(\mu\) in \(G_T\)-orbits.

\begin{lemma}
    [{\cite[Lemma 13.13]{georgoulasMomentweightInequalityHilbertMumford2021}}]
    \label{lem:relative-torus-moment-map}
    Let \(x \in X\), and suppose \(T \subseteq K_x\) is a torus. Then for all \(g \in G_T\),
    \begin{enumerate}[(i)]
        \item \(T \subseteq K_{gx}\),
        \item \(\mu(gx) \in \mfk_T\),
        \item \(\mu(gx) - \mu(x) \in \mft^\perp\).
    \end{enumerate}
\end{lemma}

Finally, we have the following result about the flow of \(-\grad f\) and the moment map \(\mu\) relative to a torus.

\begin{lemma}
    [{\cite[Lemma 13.15]{georgoulasMomentweightInequalityHilbertMumford2021}}]
    \label{lem:relative-torus-flow}
    Let \(x_0 \in X\), and let \(x(t)\) denote the flow line for \(-\grad f\) starting at \(x_0\). Let \(T \subseteq K_{x_0}\) be a torus, with Lie algebra \(\mft\). Let \(K_T, \mfk_T\) be as above. Then
    \begin{enumerate}[(i)]
        \item \(\mu(x(t)) \in \mfk_T\) for all \(t\).
        \item \(\mu(x(t)) - \mu(x_0) \in \mft^\perp\) for all \(t\).
        \item \(x(t) \in G_T \cdot x_0\) for all \(t\).
    \end{enumerate}
\end{lemma}

\section{Normal forms}

\label{sec:normal-forms}

In this section, we recall the results of \cite{chenCalabiFlowGeodesic2014}, which allow us to reduce the general case to the case of a linear action on a vector space. We defer the proofs to \cite{chenCalabiFlowGeodesic2014}, but remark that the key ingredient is the Marle-Guillemin-Sternberg normal form for Hamiltonian group actions \cite{marleModeleDactionHamiltonienne1985,guilleminNormalFormMoment1984}.

\subsection{Standard case}

Let \((V, J_0, g_0)\) be a unitary representation of a compact connected Lie group \(K\). Then \((V, J_0)\) is naturally a representation of the complexification \(G\). Let \(\Omega_0\) denote the induced K\"ahler form on \(V\). Then the \(K\) action on \(V\) has moment map
\[\langle \mu(v), \xi \rangle = \frac{1}{2}\Omega_0(v, \xi \cdot v).\]

\subsection{Linear case}

Now suppose \(K\) acts linearly on \((V, J_0, \Omega)\), where \(\Omega\) is a real-analytic symplectic form. This has a real-analytic moment map \(\mu \colon V \to \mfk\), with \(\mu(0) = 0\). On the other hand, we have a linearised \(K\)-action, on \((V = T_0V, J_0, \Omega_0)\). Let \(\widehat\mu \colon V \to \mfk\) denote the corresponding moment map.

We obtain two flow lines \(\gamma_v(t)\) and \(\widehat\gamma_v(t)\) on \(G/K\), starting from \([1]\).

\begin{proposition}
    [{\cite[Proposition 4.5]{chenCalabiFlowGeodesic2014}}]
    \label{prop:bounded-linear}
    The distance \(d(\gamma_v(t), \widehat\gamma_v(t))\) is uniformly bounded.
\end{proposition}

\subsection{General case}

Let \(X \subseteq \P^n\) be a smooth projective variety, and let \(G\) be a complex reductive Lie group, acting linearly on \(\P^n\) and preserving \(X\). Let \(K\) be a maximal compact subgroup of \(G\), acting by isometries. Let \(\mu \colon X \to \mfk\) be a moment map for the \(K\)-action.

\begin{remark}
    The assumption that \(X\) is projective can be weakened, though a real-analyticity assumption is necessary. We refer to \cite{chenCalabiFlowGeodesic2014} for details.
\end{remark}

Suppose \(y \in X\) is semistable but not polystable. Letting \(y(t)\) denote the flow line for \(-\grad f\) starting at \(y\), \(y(t)\) converges to a polystable point \(x\). The following result is classical.

\begin{lemma}
    The stabiliser \(G_x\) is reductive, with maximal compact subgroup \(K_x\).
\end{lemma}

The following result follows from a normal form theorem for the \(G\)-action near \(x\), which will be explained at the end of this section.

\begin{proposition}
    [{\cite[Proposition 4.9]{chenCalabiFlowGeodesic2014}}]
    \label{prop:move-orbit}
    There exists a point \(\widehat y \in G \cdot y\), such that \(x \in \overline{G_x \cdot \widehat y}\).
\end{proposition}

Given this, let \(\widetilde\gamma(t)\) and \(\widehat\gamma(t)\) denote the flow lines starting from \([1]\) for the Kempf-Ness function for \(G\) and \(G_x\) respectively. The space \(G_x/K_x\) is a totally geodesic submanifold of \(G/K\), and so we may regard \(\widehat\gamma(t)\) as a path in \(G/K\).

\begin{proposition}
    [{\cite[Proposition 4.10]{chenCalabiFlowGeodesic2014}}]
    \label{prop:bounded-stabiliser}
    The distance \(d(\widetilde\gamma(t), \widehat\gamma(t))\) is uniformly bounded.
\end{proposition}

Finally, we reduce to the linear case. We first explain the proof of \Cref{prop:move-orbit}. Consider the orthogonal decomposition
\[T_xX = \mfg \cdot x \oplus N_x.\]
Then \(N_x\) is a \(G_x\)-invariant complex subspace of \(T_xX\). Choose a \(K_x\)-equivariant holomorphic embedding \(\Psi \colon U \to X\), where \(U \subseteq T_xX\) is a \(K_x\)-invariant neighbourhood of \(0\), with \(\Psi(0) = x\) and \(\dd\Psi(0) = \id\). Let \(B = U \cap N_x\), \(B' = \Psi(B)\). This map can be extended uniquely to a \(G\)-equivariant map
\[\Phi \colon G \times_{G_x} S \to X\]
which is biholomorphic onto an open subset of \(X\) \cite[Theorem 1.12]{sjamaarHolomorphicSlicesSymplectic1995}. Here, \(S = G_x \cdot B\) is a \(G_x\)-invariant neighbourhood of \(0 \in N_x\).

Thus, for any \(y\) close to \(x\), there exists \((g, v) \in G \times N_x\), such that \(y = g \cdot \Phi(\id, v)\). By assumption, the Kempf-Ness flow \(y(t)\) starting from \(y\) converges to \(x\). Choosing a sequence \(t_j \to \infty\), we obtain \((g(t_j), v(t_j)) \in G \times N_x\), such that \(y(t_j) = g(t_j) \cdot \Phi(\id, v(t_j))\) and \(\lim_{j \to \infty} v(t_j) = 0\). Since \(y(t_j) \in G \cdot y\), \(v(t_j) \in G_x \cdot v\). The point \(\widehat y\) is then defined to be \(\Phi(\id, v)\).

By this construction, we have now reduced to the case of \(G_x\) acting linearly on \((N_x, \Psi^*\omega)\), and a point \(v \in N_x\) with \(0 \in \overline{G_x \cdot v}\).

\section{Finite-dimensional result}

\label{sec:finite-dimensional}

Let \(X \subseteq \P^n\) be a smooth projective variety, and let \(G\) be a complex reductive Lie group, acting on \(\P^n\) by a linear action, and preserving \(X\). Let \(\omega\) denote the restriction of the Fubini-Study metric to \(X\). Let \(K\) be a maximal compact subgroup of \(G\), acting by isometries. Let \(\mu \colon X \to \mfk\) be a moment map for the \(K\)-action. Moreover, we assume that the invariant inner product on \(\mfk\) is rational with factor \(\hbar \in \Q_{>0}\).

We let \(\log \colon G/K \to \mfk\) denote the inverse of the map \(\xi \mapsto \exp(i\xi)K\). For \(y \in X\), we have the Kempf-Ness function \(\Phi_y \colon G/K \to \R\). Let \(\gamma_y(t)\) denote the flow line for \(-\grad\Phi_y\) starting at \(\pi(1)\). We are interested in the asymptotics of \(\gamma_y(t)\) as \(t \to \infty\).

\begin{theorem}
    \label{thm:main}
    Let \(y \in X\) be semistable but not polystable. Then there exist elements \(v_1, \dots, v_k\) of \(\mfk\), such that
    \[\log(\gamma_y(t)) = \log(t)v_1 + \log\log(t)v_2 + \log\cdots\log(t)v_k + O_{t \to \infty}(1).\]
    Moreover, \([v_i, v_j] = 0\) for all \(i, j\).
\end{theorem}

To start, let \(y(t)\) denote the flow line for \(-\grad f\) starting at \(y\), and let \(x = \lim_{t \to \infty} y(t)\). By \Cref{prop:bounded-stabiliser}, we may replace \(G\) with \(G_x\) and \(K\) with \(K_x\). Next, as explained after \Cref{prop:bounded-stabiliser}, we may replace \(X\) with a neighbourhood of \(0\) in \(N_x\). Finally, by \Cref{prop:bounded-linear}, we may reduce to the standard setting. Thus, \Cref{thm:main} follows from the following result.

\begin{theorem}
    \label{thm:vector-space}
    Let \((V, J_0, g_0)\) be a unitary representation of a compact connected Lie group \(K\), which extends to a representation of the complexification \(G\). Let \(\Omega_0\) denote the induced K\"ahler form on \(V\). Let \(\mu \colon V \to \mfk\) denote the associated moment map. Let \(x \in V \setminus \{0\}\) be such that \(0 \in \overline{G \cdot x}\). Let \(\gamma_x(t)\) denote the flow line for \(-\grad f\) starting at \(x\). Then there exist elements \(v_1, \dots, v_k\) of \(\mfk\), such that
    \[\log(\gamma_x(t)) = \log(t)v_1 + \log\log(t)v_2 + \log\cdots\log(t)v_k + O_{t \to \infty}(1).\]
    Moreover, \([v_i, v_j] = 0\) for all \(i, j\).
\end{theorem}

Finally, we have an analogous result in the unstable case.

\begin{theorem}
    \label{thm:unstable}
    Let \(y \in X\) be unstable. Then there exist elements \(v_0, \dots, v_k\) of \(\mfk\), such that
    \[\log(\gamma_y(t)) = tv_0 + \log(t)v_1 + \log\log(t)v_2 + \log\cdots\log(t)v_k + O_{t \to \infty}(1).\]
    Moreover, \([v_i, v_j] = 0\) for all \(i, j\).
\end{theorem}

The proofs of \Cref{thm:main,thm:vector-space,thm:unstable} proceed by induction on \(\dim(G)\). We have shown that by reducing to a local slice, \Cref{thm:vector-space} implies \Cref{thm:main}. The proof of \Cref{thm:vector-space} relies on \Cref{thm:unstable} holding for \(X = \P(V)\). Finally, \Cref{thm:unstable} follows from \Cref{thm:main}, applied to a semistable point for a \(G_T/T^c\)-action, where \(T\) is the torus generated by the optimal destabilising direction \(v_0\). As \(\dim(G_T/T^c) < \dim(G)\), this process terminates.

Suppose we have a unitary representation \((V, J_0, g_0)\) of a compact connected Lie group \(K\), which extends to a representation of the complexification \(G\). Let \(\Omega_0\) denote the induced K\"ahler form on \(V\). Let \(\mu \colon V \to \mfk\) denote the associated moment map. 

Let \(x \in V \setminus \{0\}\) be such that \(0 \in \overline{G \cdot x}\). Let \(p \colon V \setminus \{0\} \to \P(V)\) denote the quotient map, and let \(q = p(x)\). On \((\P(V), \omega_\FS)\) with the induced action, we have a moment map
\[\widehat\mu([v]) = \frac{\mu(v)}{\abs{v}^2}.\]
We let \(\gamma_x(t)\), \(\widehat\gamma_q(s)\) denote the corresponding Kempf-Ness flows. Similarly, we define \(f = \abs{\mu}^2/2\), \(\widehat f = \abs{\widehat\mu}^2/2\). Let \(q(s)\) denote the flow line for \(-\grad \widehat f\) starting at \(q\), and let \(x(t)\) denote the flow line for \(-\grad f\) starting at \(x\).

\begin{lemma}
    We have that \(q(s) = p(x(t(s)))\), where
    \[\dv{s}{t} = \abs{x(t)}^2.\]
\end{lemma}

\begin{proof}
    The result follows from the fact that
    \[p_*(\grad f(v)) = \abs{v}^2\grad\widehat f(p(v)).\]
\end{proof}

\begin{lemma}
    \label{lem:reparametrisation}
    Under the same change of variables as in the previous lemma, we have \(\widehat\gamma_q(s) = \gamma_x(t(s)).\)
\end{lemma}

\begin{proof}
    Let \(g_x(t)\) and \(\widehat g_q(s)\) denote the corresponding flows in \(G\). Then
    \begin{align*}
        \widehat g_q^{-1}\dv{\widehat g_q}{s} &= i\widehat\mu(q(s)), \\
        &= i\frac{\mu(x(t(s)))}{\abs{x(t(s))}^2}, \\
        &= \abs{x(t(s))}^{-2} g_x^{-1}\dv{g_x}{t}.
    \end{align*}
    The result then follows by chain rule.
\end{proof}

Since \(0 \in \overline{G \cdot x}\), the point \(q\) is \(\mu\)-unstable. Thus, by \Cref{thm:kempf}, there exists \(\xi_q \in \mfk\), such that
\[\log(\widehat\gamma_q(s)) = -s\xi_q + o_{s \to \infty}(s).\]
We let
\[q' = \lim_{s \to \infty}\exp(is\xi_q) \cdot q.\]
By \Cref{lem:one-param-limit}, the limit exists, and \(L_{q'}\xi_q = 0\). By construction of the action on \(\P(V)\), we have that
\[\lim_{s \to \infty}\exp(is\xi_q) \cdot q = \lim_{s \to \infty}p(\exp(is\xi_q) \cdot x).\]

\begin{lemma}
    The direction \(\xi_q\) is rational, i.e., it generates a subgroup
    \[T = \{\exp(t\xi_q) \mid t \in \R\}\] 
    of \(K\). Moreover,
    \[\lim_{s \to \infty}\exp(is\xi_q)\cdot x = 0.\]
\end{lemma}

\begin{proof}
    Rationality follows from \Cref{prop:rationality}. For the limit, we refer to \cite[Proposition 4.4]{chenCalabiFlowGeodesic2014}.
\end{proof}

Let \(n_q > 0\) be minimal, such that \(\exp(n_q\xi_q) = 1\). Let us now consider the \(T^c\)-action on \(V\), which decomposes into weight spaces
\[V = \bigoplus_{m \in \Z}V_m,\]
where for \(v \in V_m\),
\[\exp(is n_q\xi_q) \cdot v = e^{ms} v.\]
If we write \(x = \sum_m x_m\), we must then have that \(x_m = 0\) for all \(m \ge 0\). Let \(m_0 < 0\) be maximal, such that \(x_{m_0} \neq 0\). Set \(x' = x_{m_0}\), and \(x'' = x - x'\). Then \(q' = p(x')\).

\begin{lemma}
    [{\cite[Proposition 1.9]{ramananRemarksInstabilityFlag1984}, \cite[Theorem 9.3]{nessStratificationNullCone1984}}]
    The point \(q'\) is \(\widehat\mu\)-unstable. Moreover, the direction \(\xi_q\) is the optimal destabilising direction for \(q'\).
\end{lemma}

\begin{proof}
    The fact that \(q'\) is \(\widehat\mu\)-unstable follows from the Hilbert-Mumford criterion, as \(w_{\widehat\mu}(q', \xi_q) = w_{\widehat\mu}(q, \xi_q) < 0\). Thus, it remains to show that the optimal destabilising direction for \(q'\) is \(\xi = \xi_q\).

    Let \(\eta\) denote the optimal destabilising direction for \(q'\). Let \(P(\xi), P(\eta)\) denote the corresponding parabolic subgroups of \(G\). The intersection \(P(\xi) \cap P(\eta)\) contains a maximal torus \(T\) of \(G\). This means that there exists \(g \in P(\xi), h \in P(\eta)\), such that \(\eta' = g\eta g^{-1}\) and \(\xi' = h \xi h^{-1}\) are in \(\mft\). By construction, \(\xi'\) is the optimal destabilising direction for \(r = h \cdot q\), and
    \[\lim_{s \to \infty}\exp(is\xi') \cdot r = h \cdot q' = q''.\]
    So \(q' = h^{-1} \cdot q''\). Suppose we know the result holds for \(G = T\).

    First, by parabolic invariance,
    \[w_{\widehat\mu}(q', \eta) = w_{\widehat\mu}(q', \eta'),\]
    and so the optimal destabilising direction for \(q'\) with respect to the \(T\) and \(G\)-actions are the same. On the other hand, by the result for \(G = T\), \(\xi'\) is the optimal destabilising direction for \(q''\) with respect to the \(T\)-action. But now
    \[w_{\widehat\mu}(q'', \xi') = w_{\widehat\mu}(q', \xi) = w_{\widehat\mu}(q', \xi')\]
    by equivariance and parabolic invariance. Thus, \(\xi'\) is the optimal destabilising direction for \(q'\) with respect to the \(T\)-action as well. Thus, there exists \(k \in P = P(\xi) = P(\xi')\), such that \(\eta' = k\xi' k^{-1}\). But now this means that \(P = P(\eta') = P(\eta)\) as well, and so
    \[\xi = h^{-1}\xi' h = h^{-1}k^{-1}\eta'kh = (h^{-1}k^{-1}g)\eta(h^{-1}k^{-1}g)^{-1}.\]
    But each equivalence class contains a unique element of \(\mfk\) \cite[Theorem D.4]{georgoulasMomentweightInequalityHilbertMumford2021}, and so \(\xi = \eta\), as required.

    For the torus case, we refer to \cite[Theorem 9.3]{nessStratificationNullCone1984} and \cite{szekelyhidiExtremalMetricsKstability2006}, which follows a convex-geometric argument.
\end{proof}

In particular, \(T = \{\exp(t\xi_q) \mid t \in \R\} \subseteq K_{q'}\) is a one-dimensional compact torus, and so we may now apply the theory of moment maps relative to a torus, as in \Cref{sec:relative-torus}. From \Cref{thm:kempf} and \Cref{lem:relative-torus-flow},
\[\log(\widehat\gamma_{q'}(s)) = -s\xi_q + \epsilon(s),\]
where \(\epsilon(s) = o_{s\to\infty}(s)\) satisfies \([\epsilon(s), \xi_q] = 0\).

For the proof of the next proposition, we need the following classical result about the Fubini-Study metric.
\begin{lemma}
    For \(v, w \in V\) non-zero,
    \[\cos(d(p(v), p(w))) = \frac{\abs{\langle v, w \rangle}}{\norm{v}\norm{w}}.\]
\end{lemma}

As a corollary, we have the following estimate for the Fubini-Study distance.

\begin{corollary}
    \label{cor:fs-distance-bound}
    We have the inequality \(d(p(v), p(v + w)) \le C\frac{\norm{w}}{\norm{v + w}}\). Moreover, for \(\norm{w} \le \frac{1}{2}\norm{v}\),
    \[d(p(v), p(v + w)) \le C\frac{\norm{w}}{\norm{v}}.\]
\end{corollary}

\begin{proof}
    Set \(u = v + w\), and \(u' = w - \frac{\langle v, w \rangle}{\norm{v}^2}v\), so that
    \[u = \frac{\langle v, u \rangle}{\norm{v}^2}v + u'.\]
    Taking the norm of both sides, and clearing denominators,
    \[\norm{v}^2\norm{u}^2 = \abs{\langle v, u \rangle}^2 + \norm{v}^2\norm{u'}^2.\]
    By construction, \(\norm{u'} \le \norm{w}\).
    
    Let \(d = d(p(v), p(u))\). By the lemma,
    \[\cos(d) = \frac{\abs{\langle v, u \rangle}}{\norm{v}\norm{u}},\]
    and so
    \[\sin(d)^2 = 1 - \cos(d)^2 
        = 1 - \frac{\abs{\langle v, u \rangle}^2}{\norm{v}^2\norm{u}^2} 
        = \frac{\norm{v}^2\norm{u}^2 - \abs{\langle v, u \rangle}^2}{\norm{v}^2\norm{u}^2} 
        \le \frac{\norm{w}^2}{\norm{u}^2}.\]
    The result then follows from the fact that \(\sin(d) \ge Cd\) for \(d\) small. The final result follows from \(\norm{v + w} \ge \norm{v} - \norm{w} \ge \frac{1}{2}\norm{v}\).
\end{proof}

We can now prove the following result, which states that the flow lines \(\widehat\gamma_q(s)\) and \(\widehat\gamma_{q'}(s)\) are asymptotic to each other. The first half of the proof follows the proof of \cite[Proposition 4.5]{chenCalabiFlowGeodesic2014}.

\begin{proposition}
    \label{prop:flow-one-param-limit}
    As \(s \to \infty\), we have
    \[\widehat\gamma_q(s) = \widehat\gamma_{q'}(s) + O_{s \to \infty}(1).\]
\end{proposition}

\begin{proof}
    Let \(L(s)\) denote the distance between \(\widehat\gamma_q(s)\) and \(\widehat\gamma_{q'}(s)\). We let \(\alpha_s(\tau)\) denote the geodesic from \(\widehat\gamma_q(s)\) to \(\widehat\gamma_{q'}(s)\). From this,
    \[L(s)\dv{L}{s} = \langle\alpha_s'(1), -\grad\widehat\Phi_{q'}(\widehat\gamma_{q'}(s))\rangle - \langle \alpha_s'(0), -\grad\widehat\Phi_{q}(\widehat\gamma_q(s))\rangle.\]
    Let \(p_0 = \alpha_s(0) = \widehat\gamma_q(s)\) and \(p_1 = \alpha_s(1) = \widehat\gamma_{q'}(s)\). Noting that \(\alpha_s'(1) = -\exp_{p_1}^{-1}(p_0)\) and \(\alpha_s'(0) = \exp_{p_0}^{-1}(p_1)\), the above can be rewritten as
    \begin{align*}
        L(s)\dv{L}{s} &= \langle \exp_{p_1}^{-1}(p_0), \grad\widehat\Phi_{q'}(p_1)\rangle + \langle \exp_{p_0}^{-1}(p_1), \grad\widehat\Phi_{q}(p_0)\rangle, \\
        &= \langle \exp_{p_1}^{-1}(p_0), \grad\widehat\Phi_q(p_1)\rangle + \langle \exp_{p_0}^{-1}(p_1), \grad\widehat\Phi_q(p_0)\rangle \\
        &\quad + \langle \exp_{p_1}^{-1}(p_0), \grad\widehat\Phi_{q'}(p_1) - \grad\widehat\Phi_q(p_1)\rangle.
    \end{align*}
    By the geodesic convexity of \(\widehat\Phi_q\),
    \[\langle \exp_{p_1}^{-1}(p_0), \grad\widehat\Phi_q(p_1)\rangle + \langle \exp_{p_0}^{-1}(p_1), \grad\widehat\Phi_q(p_0)\rangle \le 0.\]
    Thus,
    \[L(s)\dv{L}{s} \le \langle \exp_{p_1}^{-1}(p_0), \grad\widehat\Phi_{q'}(p_1) - \grad\widehat\Phi_q(p_1)\rangle.\]
    By Cauchy-Schwarz, and the fact that \(\abs{\exp_{p_1}^{-1}(p_0)} = L(s)\), we obtain
    \[\dv{L}{s} \le \abs{\grad\widehat\Phi_{q'}(\widehat\gamma_{q'}(s)) - \grad\widehat\Phi_q(\widehat\gamma_{q'}(s))}.\]

    Let \(\widehat g_{q'}(s)\) denote the corresponding flow in \(G\). Then
    \[\dv{L}{s} \le \abs{\widehat\mu(\widehat g_{q'}(s)^{-1} \cdot q') - \widehat\mu(\widehat g_{q'}(s)^{-1} \cdot q)}.\]
    Since \(\widehat\mu \colon \P(V) \to \mfk\) is a smooth function on a compact manifold, it is Lipschitz. Thus,
    \[\dv{L}{s} \lesssim d(\widehat g_{q'}(s)^{-1} \cdot q', \widehat g_{q'}(s)^{-1} \cdot q).\]
    It thus suffices to show that
    \[d(\widehat g_{q'}(s)^{-1} \cdot q', \widehat g_{q'}(s)^{-1} \cdot q) \lesssim e^{-\delta s}\]
    for some \(\delta > 0\).

    We have that \(\widehat g_{q'}(s) = \exp(-is\xi_q + \epsilon(s))u(s)\), where \(u(s)\) is a path in \(K\), and \(\epsilon(s) = o_{s\to\infty}(s)\), with \([\xi_q, \epsilon(s)] = 0\). We may rewrite the above bound as
    \[d(p(\widehat g_{q'}(s)^{-1} \cdot x'), p(\widehat g_{q'}(s)^{-1} \cdot x)) \lesssim e^{-\delta s}.\]
    By \Cref{cor:fs-distance-bound}, and using the fact that \(x = x' + x''\), it suffices to show that
    \[\frac{\norm{\widehat g_{q'}(s)^{-1} \cdot x''}}{\norm{\widehat g_{q'}(s)^{-1} \cdot x'}} \lesssim e^{-\delta s}.\]
    Since \(u(s)\) is unitary, this just becomes
    \[\frac{\norm{\exp(is\xi_q - \epsilon(s)) \cdot x''}}{\norm{\exp(is\xi_q - \epsilon(s)) \cdot x'}} \lesssim e^{-\delta s}.\]

    Let us arrange the eigenvalues of \(i\xi_q\) acting on \(V\) as
    \[\lambda_1(i\xi_q) \ge \cdots \ge \lambda_n(i\xi_q), n = \dim V.\]
    Suppose that \(\lambda_1(i\xi_q), \dots, \lambda_r(i\xi_q) \ge 0\), and \(\lambda_{r+1}(i\xi_q), \dots, \lambda_n(i\xi_q) < 0\). Let \(V_i\) denote the eigenspace with eigenvalue \(\lambda_i(i\xi_q)\). Then
    \[x \in \bigoplus_{i=r+1}^n V_i.\]
    Letting
    \[\nu = \max\{\lambda_i \mid x_i \ne 0\} < 0,\]
    we set
    \[x' = \sum_{i \text{ s.t. }\lambda_i(i\xi_q) = \nu}x_i,\]
    and
    \[x'' = \sum_{i \text{ s.t. }\lambda_i(i\xi_q) < \nu}x_i.\]
    Let
    \[\delta = \nu - \max\{\lambda_i \mid x_i \ne 0, \lambda_i < \nu\} > 0.\]
    Since \([\xi_q, \epsilon(s)] = 0\), it follows that
    \[\exp(is\xi_q - \epsilon(s))\cdot x' = \exp(-\epsilon(s))e^{\nu s}x'.\]
    In particular, 
    \[\norm{\exp(is\xi_q - \epsilon(s)) \cdot x'} \ge e^{\nu s + \lambda_1(\epsilon(s))}\norm{x'} \gtrsim e^{\nu s}\norm{x'}.\]
    The same argument applies to \(x''\), to give an upper bound
    \[\norm{\exp(is\xi_q - \epsilon(s)) \cdot x''} \lesssim e^{(\nu - \delta)s}\norm{x''}.\]
    Combining these two bounds, we obtain the desired result.
\end{proof}

\begin{proposition}
    \label{prop:relative-torus-flow}
    For all \(s\), we have
    \(\widehat\gamma_{q'}(s) = -s\xi_q + \widehat\gamma_{G_T/T^c, q'}(s)\).
\end{proposition}

\begin{proof}
    For this, we observe that the inner product gives a canonical splitting
    \[\mfk_T = \mft \oplus \mfk_T/\mft.\]
    Moreover, by \Cref{lem:relative-torus-flow}, if \(q'(s)\) is the moment map flow from \(q'\), then
    \[\widehat\mu(q'(s)) \in \mfk_T\]
    for all \(s\). Moreover, the \(\mft\) component is constant, accounting for the \(-s\xi_q\) term. 
    
    Next, the \(\mfk_T/\mft\) component of \(\widehat \mu(q')\) is, under the orthogonal splitting, the moment map for the \(K_T/T\)-action on \(((\P(V))_T, \omega)\). Thus, if we let \(\widetilde q'(s)\) denote the flow for the \(K_T/T\)-action, starting from \(q'\), then \(q'(s) = \widetilde q'(s)\) for all \(s\), and this accounts for the \(\widehat\gamma_{G_T/T^c, q'}(s)\) term.
\end{proof}

\begin{proposition}
    \label{prop:relative-torus-semistable}
    The point \(q' \in \P(V)_T\) is semistable with respect to the \(G_T/T^c\)-action on \((\P(V)_T, \omega)\).
\end{proposition}

\begin{remark}
    The same result has been proven by Ramanan-Ramanathan \cite{ramananRemarksInstabilityFlag1984}, Kirwan \cite{kirwanCohomologyQuotientsSymplectic1984}, Bruasse-Teleman \cite{bruasseHarderNarasimhanFiltrationsOptimal2005} and Ness \cite[Theorem 9.4]{nessStratificationNullCone1984} in different levels of generality. We include the proof for completeness, and because the result is central to our inductive argument.
\end{remark}

\begin{proof}
    We will apply the Hilbert-Mumford criterion (\Cref{thm:hilbert-mumford}). Recall that the toral generators in \(\mfg_T/\mft^c\) are given by
    \[\mcT^c_{\mfg_T/\mft^c} = \{\zeta + \mft^c \mid \zeta \in \mcT^c \cap \mfg_T \setminus \mft\}.\]
    By \Cref{lem:relative-torus-weight}, for \(\zeta \in \mcT^c \cap \mfg_T \setminus \mft\),
    \[w_{\mu_T}(q', \zeta + \mft^c) = w_{\widehat\mu, T}(q', \zeta) = w_{\widehat\mu}(q', \zeta) - \langle \widehat\mu(q'), \pr_T(\zeta)\rangle.\]
    Thus, it suffices to show that
    \[w_{\widehat\mu}(q', \zeta) \ge \langle {\widehat\mu}(q'), \pr_T(\zeta) \rangle.\]
    
    We write \(\zeta = a\xi_q + \eta\), where \(a \in \R\) and \(\eta \in \mft^\perp\). As \([\eta, \xi_q] = 0\), and that \(L_q'\xi_q = 0\), it follows that
    \[\exp(is\zeta)q' = \exp(is\eta)q'.\]
    Moreover, by \Cref{lem:relative-torus-moment-map},
    \[\widehat\mu(\exp(is\eta)q') - \widehat\mu(q') \perp \xi_q.\]
    Let \(q'' = \lim_{s \to \infty}\exp(is\eta)q'\). Thus,
    \begin{align*}
        w_{\widehat\mu}(q', \zeta) &= \lim_{s \to \infty}\langle\widehat\mu(q''), \zeta\rangle \\
        &= \langle \widehat\mu(q'') - \widehat\mu(q'), \zeta\rangle + \langle \widehat\mu(q'), \zeta\rangle, \\
        &= \langle \widehat\mu(q''), \eta\rangle - \langle \widehat\mu(q'), \eta\rangle + a\langle \widehat\mu(q'), \xi_q\rangle + \langle \widehat\mu(q'), \eta\rangle, \\
        &= -a\abs{\xi_q}^2 + \langle\widehat\mu(q''), \eta\rangle, \\
        &= \langle \widehat\mu(q'), \pr_T(\zeta) \rangle + w_{\widehat\mu}(q', \eta).
    \end{align*}
    Fix \(\eta \in \mfk_T \cap \mft^\perp\), and now consider \(\zeta_t = \xi_q + t\eta\), for \(t > 0\). The above computation shows that
    \[w_{\widehat\mu}(q', \zeta_t) = -\abs{\xi_q}^2 + tw_{\widehat\mu}(q', \eta).\]
    Moreover, \(\abs{\zeta_t}^2 = \abs{\xi_q}^2 + t^2\abs{\eta}^2\). Then
    \[\varphi(t) = \frac{w_{\widehat\mu}(q', \zeta_t)}{\abs{\zeta_t}} = \frac{-\abs{\xi_q}^2 + tw_{\widehat\mu}(q', \eta)}{\sqrt{\abs{\xi_q}^2 + t^2\abs{\eta}^2}} = \frac{-m^2 + ta}{\sqrt{m^2 + t^2b^2}}.\]
    On the other hand, the moment weight inequality \Cref{thm:moment-weight} implies that
    \[\frac{w_{\widehat\mu}(q', \zeta_t)}{\abs{\zeta_t}} \ge \frac{w_{\widehat\mu}(q', \xi_q)}{\abs{\xi_q}} = -m.\]
    We claim that if \(a = w_{\widehat\mu}(q', \eta) < 0\), then there exists \(t > 0\) such that \(\varphi(t) < -m\), which would contradict \Cref{thm:moment-weight}. For this, observe that
    \[\varphi'(t) = \frac{m^2(a + b^2t)}{(m^2 + t^2b^2)^{3/2}},\]
    and so
    \[\varphi(0) = -m \quad\text{and}\quad\varphi'(0) = \frac{a}{m}.\]

    Thus, we must have that \(w_{\widehat\mu}(q', \eta) \ge 0\), and so
    \[w_{\widehat\mu}(q', \zeta) \ge \langle \widehat\mu(q'), \pr_T(\zeta) \rangle\]
    as required.
\end{proof}

We may now prove \Cref{thm:unstable}.

\begin{proof}
    [Proof of \Cref{thm:unstable}]
    First of all, we note that despite being stated for \(X = \P(V)\), the proofs of \Cref{prop:relative-torus-flow} and \Cref{prop:relative-torus-semistable} hold for any smooth projective variety. The result then follows immediately from \Cref{prop:relative-torus-flow}, \Cref{prop:relative-torus-semistable} and \Cref{thm:main} for the \(G_T/T^c\)-action on \(X_T\).
\end{proof}

Thus, it remains to prove \Cref{thm:vector-space} from \Cref{thm:unstable}. For this, we will need to understand the reparametrisation \(s(t)\) in \Cref{lem:reparametrisation}.

\begin{proposition}
    \label{prop:reparametrisation}
    There exists a constant \(C > 0\), such that
    \[s(t) = C\log(t) + O_{t\to\infty}(1).\]
\end{proposition}

To prove the result, we begin by establishing some preliminary estimates on the growth of \(s(t)\). Define
\begin{align*}
    r(t) &= \abs{g_x(t)^{-1} \cdot x}^2, \\
    \widehat r(s) &= \abs{\widehat g_q(s)^{-1} \cdot x}^2, \\
    F(t) &= f(g_x(t)^{-1} \cdot x), \\
    \widehat F(s) &= \widehat f(\widehat g_q(s)^{-1} \cdot q).    
\end{align*}
Recall that \(\dv{s}{t} = r(t)\).

\begin{lemma}
    For all \(t\), we have
    \[\dv{r}{t} = -4F(t).\]
\end{lemma}

\begin{proof}
    This follows from the definition of the gradient flow, and the fact that \(f \colon V \to \R\) is a homogeneous polynomial of degree 4. 
\end{proof}

Next, notice that
\[\widehat F(s(t)) = \frac{F(t)}{r(t)^2},\]
and that the limit
\[\lim_{s \to \infty} \widehat F(s) = L\]
exists and is non-zero.

\begin{lemma}
    There exists \(C > 0\), such that for \(s\) sufficiently large,
    \[\frac{1}{C}\log(t) \le s(t) \le C \log(t).\]
\end{lemma}

\begin{proof}
    As \(\widehat F(s)\) is non-increasing, it follows that \(0 < L \le \widehat F(s(t)) \le C\) for some \(C > 0\). Thus,
    \[-4Lr(t)^2 \ge \dv{r}{t} \ge -4C r(t)^2.\]
    Setting \(u(t) = 1/r(t)\), this implies that
    \[4L \le \dv{u}{t} \le 4C,\]
    and so \(u(t) \asymp t\). Thus, \(r(t) \asymp 1/t\), and so \(s(t) \asymp \log(t)\).
\end{proof}

\begin{proof}
    [Proof of \Cref{prop:reparametrisation}]
    By the {\L}ojasiewicz inequality, there exists \(\gamma = 1/(2\alpha - 1) > 1\) such that
    \[\abs{\widehat F(s) - L} \lesssim s^{-\gamma} \lesssim \log(t)^{-\gamma}.\]
    Set \(\epsilon(t) = \widehat F(s(t)) - L\), and \(u(t) = 1/r(t)\). Then
    \[\dv{u}{t} = 4u(t)^2 F(t) = 4 \widehat F(s(t)) = 4L + 4\epsilon(t).\]
    We can decompose
    \begin{align*}
        \abs{\int_{t_0}^t\epsilon(\tau)\dd\tau} &\lesssim \int_{t_0}^t\log(\tau)^{-\gamma}\dd\tau, \\
        &= \int_{t_0}^{\sqrt t}\log(\tau)^{-\gamma}\dd\tau + \int_{\sqrt t}^t\log(\tau)^{-\gamma}\dd\tau.
    \end{align*}
    For the first term, on \([t_0, \sqrt t]\), we have \(\log(\tau) \ge \log(t_0)\), and so \(\log(\tau)^{-\gamma} \le \log(t_0)^{-\gamma}\). Thus,
    \[\int_{t_0}^{\sqrt t}\log(\tau)^{-\gamma}\dd\tau \lesssim \sqrt t.\]
    For the second term, on \([\sqrt t, t]\), we have \(\log(\tau) \ge \log(\sqrt t) = \frac{1}{2}\log(t)\), and so \(\log(\tau)^{-\gamma} \le 2^\gamma\log(t)^{-\gamma}\). Thus,
    \[\int_{\sqrt t}^t\log(\tau)^{-\gamma}\dd\tau \lesssim t\log(t)^{-\gamma}.\]
    Combining these two bounds, \(\epsilon(t) \lesssim t\log(t)^{-\gamma}\). Thus,
    \[u(t) = 4Lt(1 + O_{t\to\infty}(\log(t)^{-\gamma})).\]
    Inverting this, we obtain
    \[r(t) = \frac{1}{4L}t^{-1} + \theta(t),\]
    where \(\theta(t) = O_{t\to\infty}(t^{-1}\log(t)^{-\gamma})\). Now
    \begin{align*}
        \abs{\int_{t_0}^t\theta(\tau)\dd\tau} &\lesssim \int_{t_0}^t\tau^{-1}\log(\tau)^{-\gamma}\dd\tau, \\
        &= \int_{\log(t_0)}^{\log(t)}u^{-\gamma}\dd u, \\
        &\lesssim \log(t)^{1-\gamma} \lesssim 1.
    \end{align*}
    Thus,
    \[s(t) = C\log(t) + O_{t\to\infty}(1).\]
\end{proof}

We may now prove \Cref{thm:vector-space}.

\begin{proof}
    [Proof of \Cref{thm:vector-space}]
    
    First, \Cref{lem:reparametrisation} implies that
    \[\gamma_x(t) = \widehat\gamma_q(s(t)),\]
    for some function \(s(t)\). \Cref{prop:reparametrisation} tells us that
    \[s(t) = C\log(t) + O_{t\to\infty}(1).\]
    This means that \(O_{s\to\infty}(1) = O_{t\to\infty}(1)\), and we write this as \(O(1)\). 
    The result then follows from \Cref{thm:unstable}.

    \end{proof}

\section{Iterated balancing filtration}

\label{sec:iterated-balancing-filtration}

In \cite{ibaneznunezRefinedHarderNarasimhanFiltrations2024}, Ib\'a\~nez N\'u\~nez defines a notion of an \emph{iterated balancing filtration} for a good moduli stack \cite{alperGoodModuliSpaces2013}. In the setting of GIT, let \(X\) be a separated scheme of finite type over \(\C\), and \(G\) a reductive linear algebraic group over \(\C\) acting on \(X\). Suppose that the stack \([X/G]\) admits a good moduli space, and that \(G\) has a norm on cocharacters \cite[Definition 2.3.8]{ibaneznunezRefinedHarderNarasimhanFiltrations2024}. Associated to a \(\C\)-point \(x \in X(\C)\), we obtain a sequence \(\lambda_1, \dots, \lambda_k\) of commuting one-parameter subgroups, through the iterated balancing filtration, as we shall see in \Cref{subsec:iterated-balancing-filtration-construction}. For each one-parameter subgroup \(\lambda_i\), we obtain a rational element
\[\eta_i = \dv{t}\bigg\vert_{t=1}\lambda_i(t) \in \mfk.\]
Note that the \(\lambda_i\) are only defined up to conjugation by the corresponding parabolic subgroups, and we make the choice such that \(\lambda_i(S^1) \subseteq K\). We refer to \cite[Remark 3.2.16]{ibaneznunezRefinedHarderNarasimhanFiltrations2024} for further details.

The goal of this section is to prove the following result.

\begin{theorem}
    \label{thm:iterated-balancing-filtration}
    Let \(v_j\) be the elements of \(\mfk\) in \Cref{thm:main} (resp.\ \Cref{thm:vector-space}). Then there exists \(C_j > 0\), such that
    \[v_j = C_j \eta_j.\]
\end{theorem}

\begin{remark}
    The constants \(C_j\) arise from the various normalisations in this paper and in \cite{ibaneznunezRefinedHarderNarasimhanFiltrations2024}. They can be computed explicitly, but we do not do so in this paper.
\end{remark}

In \Cref{subsec:iterated-balancing-filtration-construction}, we will recall the construction of the iterated balancing filtration, and then in \Cref{subsec:proof-iterated-balancing-filtration}, we will prove \Cref{thm:iterated-balancing-filtration}.

\subsection{Construction of the iterated balancing filtration}

\label{subsec:iterated-balancing-filtration-construction}

We will follow the construction in \cite[Section 4]{ibaneznunezRefinedHarderNarasimhanFiltrations2024}, where Ib\'a\~nez N\'u\~nez defines a notion of a \emph{chain of stacks}. We will restrict ourselves to the case when \(X \subseteq \P^n\) is a smooth projective variety, where \(G\) acts linearly on \(\P^n\) preserving \(X\). As part of the definition of the iterated balancing filtration, we will need a \emph{norm} on the quotient stack \([X/G]\) \cite[Definition 4.1.12]{halpernleistner2022structureinstabilitymodulitheory}. In our setting, this will be given by a \emph{norm on cocharacters} of \(G\), which we will now recall the definition of.

\begin{definition}
    [{\cite[Definition 1.3]{hesselinkUniformInstabilityReductive1978}}]
    Let \(G\) be a smooth affine algebraic group over \(\C\). A \emph{norm on cocharacters} of \(G\) is a map \(q \colon \Gamma(G) \to \Q\), which is invariant under the action of \(G(\C)\) on \(\Gamma(G)\) by conjugation, and such that for any torus \(T\) of \(G\), the restriction of \(q\) to \(\Gamma(T)\) is the quadratic form associated to a positive definite inner product on \(\Gamma(T)\).
\end{definition}

Here, \(\Gamma(G)\) denotes the set of one-parameter subgroups \(\lambda \colon \C^* \to G\) of \(G\). This is equivalent to a Weyl-invariant rational inner product on \(\Gamma(T)\) \cite[\S 1.4]{hesselinkUniformInstabilityReductive1978}. In our setting, one is given by fixing a rational \(K\)-invariant inner product on \(\mfk\) \cite[\S 12.4]{kirwanCohomologyQuotientsSymplectic1984}.

Let \(y \in X^{\mathrm{ss}}(\C)\) be a semistable point, such that the moment map flow converges to \(x \in X^{\mathrm{ps}}(\C)\). We let
\[\phi \colon X^{\mathrm{ss}} \to X \sslash G\]
denote the GIT quotient map, and
\[\pi \colon \mcX = [X^{\mathrm{ss}}/G] \to X \sslash G\]
the corresponding good moduli space. We let \(\mcF = \pi^{-1}\pi(x)\) denote the fibre of \(\pi\) over \(\pi(x)\), which is given by \(\mcF = [F/G]\), where \(F = \phi^{-1}\phi(x)\).

\begin{lemma}
    [{\cite[Theorem 1.2]{alperLunaEtaleSlice2020}}, {\cite[Proposition 7.4.1]{ibaneznunezRefinedHarderNarasimhanFiltrations2024}}] 
    \label{lem:stack-luna-slice}
    
    There is a pointed closed immersion \(\iota \colon (\mcF, x) \to (\mcN_x, 0)\), where \(\mcN_x = [N_x/G_x]\).
\end{lemma}

As the balancing stratification is compatible with pullback under closed immersions \cite[Proposition 3.5.5]{ibaneznunezRefinedHarderNarasimhanFiltrations2024}, we may now replace \(y \in F \subseteq X^{\mathrm{ss}}(\C)\) with a point \(v \in N_x\), with \(0 \in \overline{G_x \cdot v}\).

Next, we consider the blow-up 
\[b \colon \mcM = \Bl_{\mcN_x^{\mathrm{max}}}\mcN_x \to \mcN_x,\]
where \(\mcN_x^{\mathrm{max}}\) is the closed substack of \(\mcN_x\) consisting of points with maximal stabiliser dimension, with the stack structure defined by Edidin-Rydh \cite[Appendix C]{edidinCanonicalReductionStabilizers2021}. Let \(G_x^0\) denote the identity component of \(G_x\). By reductivity, we may decompose
\[N_x = U \oplus V,\]
where \(U\) is the sum of the trivial \(G_x^0\)-subrepresentations, and \(V\) the sum of the non-trivial \(G_x^0\)-subrepresentations. 

\begin{lemma}
    The point \(v\) is in \(V\).
\end{lemma}

\begin{proof}
    This follows from the fact that \(0 \in \overline{G_x \cdot v}\), and that \(G_x/G_x^0\) is finite.
\end{proof}

In this case, \(\mcN_x^{\mathrm{max}} = [U/G_x]\), and so
\[\mcM = \Bl_{[U/G_x]}[N_x/G_x] = [\Bl_U N_x / G_x].\]
As \(0 \in \overline{G_x \cdot v}\), the point \(v\) is not in \(\mcN_x^{\mathrm{max}}\), and so lifts to a point \(\widehat v \in \mcM \setminus \mcE\), where \(\mcE = [(U \times \P(V))/G_x]\) denotes the exceptional divisor. The polarisation \(\mcO(-\mcE)\) induces a \(\Theta\)-stratification of \(\mcM\), and we let \(\mcS_c\) denote the unique \(\Theta\)-stratum containing \(\widehat v\). In particular, this defines a Harder-Narasimhan filtration \(f\) for \(\widehat v\). Letting \(\mcZ\) denote the centre of \(\mcS_c\), we then obtain a point \(\gr f \in \mcZ(\C)\). The process then proceeds by induction, with the point \(\gr f \in \mcZ(\C)\).

We will conclude this discussion by describing the \(\Theta\)-stratification in terms of GIT. The point \(\widehat v \in \Bl_U N_x \setminus E\) is unstable, with respect to the polarisation \(\mcO(-E)\). Thus, by Kempf, there exists an optimal destabilising direction \(\xi \in \mfk_x\), which generates a one-dimensional torus \(T_\xi\). We let
\[w = \lim_{t \to \infty}\exp(it\xi)\cdot \widehat v,\]
which is a semistable point for the \((G_x)_{T_\xi} / T_\xi^c\)-action on \((\Bl_U N_x)_{T_\xi}\), as in \Cref{prop:relative-torus-semistable}. A direct computation proves the following.

\begin{lemma}
    [{\cite[Lemma 7.6]{kirwanPartialDesingularisationsQuotients1985}}]
    The point \(w\) lies in the exceptional divisor \(E\).
\end{lemma}

Thus, we obtain a point \(w \in \P(V)_{T_\xi}\), which is semistable with respect to the \((G_x)_{T_\xi}/T_\xi^c\)-action.

\subsection{Proof of \Cref{thm:iterated-balancing-filtration}}

\label{subsec:proof-iterated-balancing-filtration}

We now prove \Cref{thm:iterated-balancing-filtration}, by passing to a slice. First, we will show that the point \(v \in N_x\) obtained in \Cref{lem:stack-luna-slice} agrees with the one constructed in \Cref{sec:normal-forms}. By Luna's \'etale slice theorem \cite{lunaSlicesEtales1973}, we have \(G\)-equivariant morphisms
\[\begin{tikzcd}
	{(G \times_{G_x}N_x, [1, 0])} & {(G \times_{G_x}Y, [1, x])} & (X, x).
	\arrow["\phi"', from=1-2, to=1-1]
	\arrow["\sigma", from=1-2, to=1-3]
\end{tikzcd}\]
which are \'etale at \([1, x]\). We observe that \(\phi\) arises from a \(G_x\)-equivariant morphism \(Y \to N_x\) which is \'etale at \(x\). Thus, in the \emph{analytic} category, \(\phi\) is locally invertible, say with local inverse \(\psi\). One may then use \(\sigma \circ \psi\) to define an analytic slice, following Sjamaar \cite{sjamaarHolomorphicSlicesSymplectic1995}. Once we fix these choices, the resulting points in \(N_x\) are the same.

Thus, we have now reduced to the case of a representation. Let \(G\) be a reductive linear algebraic group, \(W\) a representation of \(G\). Let \(W = U \oplus V\), where \(U\) is the sum of the trivial \(G\)-subrepresentations, and \(V\) the sum of the non-trivial \(G\)-subrepresentations. Let \(v \in V\) be a point such that \(0 \in \overline{G \cdot v}\).

\begin{lemma}
    The optimal destabilising directions for
    \begin{enumerate}[(i)]
        \item \(\widehat v \in \Bl_U V\), with respect to the polarisation \(\mcO(-E)\),
        \item \([v] \in \P(V)\), with respect to the polarisation \(\mcO(1)\),
        \item \([v] \in \P(W)\), with respect to the polarisation \(\mcO(1)\)
    \end{enumerate}
    agree.
\end{lemma}

\begin{proof}
    The equality in (ii) and (iii) follows from computing the Hilbert-Mumford weights, or by observing that \(\P(V)\) is a \(G\)-invariant subvariety of \(\P(W)\). Next, observe that we have a natural \(G\)-equivariant isomorphism
    \[\Bl_U W \cong U \times \Bl_0 V.\]
    As \(G\) acts trivially on \(U\), we may assume without loss of generality that \(U = 0\), and that \(V = W\).

    Let us embed \(\Bl_0V\) into \(V \times \P(V)\), and we let \(p \colon \Bl_0V \to \P(V)\) denote the projection. Then the polarisation \(\mcO(-E)\) is given by \(p^*\mcO(1)\). Then for any one-parameter subgroup \(\lambda\) of \(G\), we have that
    \[\mu_{\mcO(-E)}(\widehat v, \lambda) = \mu_{p^*\mcO(1)}(\widehat v, \lambda) = \mu_{\mcO(1)}(p(\widehat v), \lambda) = \mu_{\mcO(1)}([v], \lambda),\]
    by functoriality properties of the Hilbert-Mumford weight \cite[49]{mumfordGeometricInvariantTheory1994}.
\end{proof}

Thus, the proof of \Cref{thm:iterated-balancing-filtration} follows by induction. At each step, the optimal destabilising directions agree. In \Cref{sec:finite-dimensional}, we obtain as the limit for the one-parameter subgroup a point \(q \in \P(N_x)_{T_\xi}\). The above argument shows that in fact the point \(q\) lies in the \((G_x)_{T_\xi}/T_\xi^c\)-invariant subvariety \(\P(V)_{T_\xi} \subseteq \P(N_x)_{T_\xi}\), and is a semistable point.

\section{Quiver representations}

\label{sec:quiver-representations}

Let \(Q = (Q_0, Q_1)\) be a quiver, \(d \in \Z_{\ge 0}^{Q_0}\) a dimension vector. Let
\[X = \Rep(Q, d) = \bigoplus_{(i \to j) \in Q_1} \Hom(\C^{d_i}, \C^{d_j}),\]
denote the space of representations of \(Q\) with dimension vector \(d\). Let
\[G = \prod_{i \in Q_0}\GL(d_i, \C)\]
act on \(X\), via
\[(g_i) \cdot (A_{ij}) = g_j A_{ij} g_i^{-1}.\]
Let \(K = \prod_{i \in Q_0}U(d_i)\) denote the maximal compact subgroup of \(G\).

Let \(\mu : X \to \mfk\) denote the canonical moment map, for the linear action of \(K\) on \(X\). However, when \(Q\) does not have oriented cycles, \(\mu\)-stability is not the correct notion. In general, one needs to include a character \(\chi \colon G \to \C^*\) \cite{kingModuliRepresentationsFinitedimensional1994}, which corresponds to shifting the moment map by a central element of \(\mfk\).

In \cite{ibaneznunezRefinedHarderNarasimhanFiltrations2024}, Ib\'a\~nez N\'u\~nez introduces the notion of a \emph{linearly lit stack}, for which there is an artinian lattice attached to each geometric point. We specialise to the setting of quiver representations. Fix for each \(i \in Q_0\) a rational number \(\theta_i \in \Q\), and \(m_i \in \Q_{> 0}\). These correspond to the character and norm on \(G\) respectively \cite[Section 3.6.3]{ibaneznunezRefinedHarderNarasimhanFiltrations2024} for the iterated balancing filtration, and to a central element of \(\mfk\) and an inner product on \(\mfk\) respectively for the moment map. Suppose \(\phi\) is a \(\theta\)-semistable representation of \(Q\).

\begin{theorem}
    [{Special case of \cite[Theorem 1.6.1]{ibaneznunezRefinedHarderNarasimhanFiltrations2024}}] Under a canonical bijection, the iterated balancing filtration of \(\phi \in [X/G]\) agrees with the HKKP filtration of the lattice of subrepresentations of \(\phi\) with the same slope.
\end{theorem}

As a corollary of this, along with \Cref{thm:vector-space,thm:iterated-balancing-filtration}, we recover the following result of Haiden-Katzarkov-Kontsevich-Pandit \cite{haidenSemistabilityModularLattices2023}. The negative gradient flow for the hermitian metric is given by
\[m_i h_i^{-1}\dv{h_i}{t} = \sum_{\alpha \colon i \to j}h_i^{-1}\phi_\alpha^* h_j\phi_\alpha - \sum_{\alpha \colon j \to i}\phi_\alpha h_j^{-1}\phi_\alpha^*h_i - \theta_i.\]
We set \(h = (h_i)_{i \in Q_0} \in G/K\).

\begin{corollary}
    [{\cite[Theorem 5.11]{haidenSemistabilityModularLattices2023}}]
    \label{cor:hkkp-quiver}
    There exists \(v_1, \dots, v_k\), such that
    \[\log(h(t)) = \log(t)v_1 + \dots + \log \cdots \log(t)v_k + O(1),\]
    where \(v_1, \dots, v_k\) are commuting elements of \(\mfk\). Moreover, the \(v_1, \dots, v_k\) can be computed by the HKKP filtration of the lattice of subrepresentations of \(\phi\) with the same slope.
\end{corollary}

\begin{remark}
    The techniques used in this paper are very different from those used in \cite{haidenSemistabilityModularLattices2023}, where a key ingredient is the monotonicity formula for the flow \cite[Proposition 5.4]{haidenSemistabilityModularLattices2023}, which also holds for the Yang-Mills flow \cite[Proposition 4.3]{haidenIteratedLogarithmsGradient2018}. We also refer to \cite{haidenIteratedLogarithmsGradient2018} for a more algebraic approach to many of the objects appearing in this paper, stated in the language of \emph{lozenge algebras}.
\end{remark}

\section{Yang-Mills flow}

\label{sec:ym}

In this section, we apply \Cref{thm:vector-space} to the Hermitian-Yang-Mills flow, by reducing to the finite-dimensional setting, following the work of Chen-Sun \cite{chenCalabiFlowGeodesic2014} for the Calabi flow, which we also use in \Cref{sec:calabi-flow}.

Let \((X, \omega)\) be compact K\"ahler, \(E \to X\) a smooth complex vector bundle. Fix a hermitian metric \(h\) on \(E\), and suppose that the Dolbeault operator \(\delbar_0\) is integrable, with \((E, h, \delbar_0)\) Hermitian-Yang-Mills. 

First, we will discuss the moment map interpretation of the Hermitian-Yang-Mills equation, which is due to Atiyah-Bott \cite{atiyahYangMillsEquationsRiemann1983} and Donaldson \cite{donaldsonSelfdualYangMillsConnections1985}. Let \(\mcK = U(E, h)\) denote the unitary gauge group, and \(\mcG = \GL(E)\) the complex gauge group. Let \(\mcD = \mcD(E)\) denote the space of (not necessarily integrable) Dolbeault operators on \(E\). On \(\mcD\), there exists a symplectic form \(\Omega\), and an almost complex structure \(J\), making \((\mcD, \Omega, J)\) into an infinite-dimensional K\"ahler manifold. The \(\mcK\)-action is by holomorphic isometries, and has a moment map
\begin{align*}
    \mu \colon \mcD &\to \Lie(\mcK) \\
    \delbar &\mapsto \Lambda F_{\delbar} + ic\id,
\end{align*}
where \(c\) is a cohomological constant. Hermitian-Yang-Mills connections correspond to zeroes of the moment map, and the Yang-Mills flow is the associated moment map flow. The corresponding flow in \(\mcG/\mcK\) is
\[h^{-1}\pdv{h}{t} = -2\left(i\Lambda_\omega F_{\delbar, h} - c\id\right).\]
\begin{theorem}
    [{\cite{donaldsonSelfdualYangMillsConnections1985}}]
    The flow \(h(t)\) exists for all time.
\end{theorem}

Next, there is an elliptic complex
\[\begin{tikzcd}
	{A^0(X, \End E)} & {A^{0,1}(X, \End E)} & {A^{0,2}(X, \End E)}
	\arrow["{\delbar_0}", from=1-1, to=1-2]
	\arrow["{\delbar_0}", from=1-2, to=1-3]
\end{tikzcd}\]
which controls the deformation theory of \((E, h, \delbar_0)\). Let \(K = U(E, h) \cap \Aut(E, \delbar_0)\) and \(G = \Aut(E, \delbar_0)\). Then \(G\) is reductive, with maximal compact subgroup \(K\). Let \(\Delta = \delbar_0^*\delbar_0 + \delbar_0\delbar_0^*\) denote the Laplacian, and \(H^1 = \ker(\Delta) = H^{0, 1}(X, \End E)\) the space of harmonic \((0, 1)\)-forms. On \(H^1\), we have a linear \(G\)-action by conjugation.

Using a \(\mcK\)-invariant metric on \(\Lie(\mcK)\), we can define a projection \(p \colon \Lie(\mcK) \to \Lie(K)\), and \(p \circ \mu \colon \mcD \to \Lie(K)\) is the moment map for the \(K\)-action on \(\mcD\). Let us now fix \(\delbar \in \mcD\) such that the Yang-Mills flow converges to \(\delbar_0\). In particular, \(\delbar_0 \in \overline{\mcG \cdot \delbar}\). We set \(\mcE = (E, \delbar)\).

Finally, the Hermite-Einstein metric \(h\) on \(\mcE\) is not unique, and we assume that \(h\) and \(\omega\) are chosen such that the induced \(K\)-invariant inner product on \(\mfk = \Lie(K)\) is rational.

\begin{theorem}
    \label{thm:ym}
    Let \(h(t)\) denote the solution to the Yang-Mills flow on \(\mcE\). Then there exists \(v_1, \dots, v_k\), such that
    \[\log(h(t)) = \log(t)v_1 + \dots + \log \cdots \log(t)v_k + O(1),\]
    where \(v_1, \dots, v_k\) are commuting elements of \(\mfk\). Moreover, the \(v_1, \dots, v_k\) can be computed by the HKKP filtration of the lattice of subbundles of \(\mcE\) with the same slope.
\end{theorem}

When \(X\) is a compact Riemann surface, this is a result of Haiden-Katzarkov-Kontsevich-Pandit \cite[Theorem 1.1]{haidenIteratedLogarithmsGradient2018}. Their approach uses the formalism of \emph{lozenge algebras}, and is also motivated by the deformation theory of holomorphic vector bundles.

\begin{remark}
    \label{rmk:ym-lattice}
    As explained in \cite[Section 4.1]{haidenIteratedLogarithmsGradient2018} and \cite[Section 1.6]{huybrechtsGeometryModuliSpaces2010}, in general one should consider the lattice of subobjects in \(\Coh_{(1)}^\phi(X)\). Here, we let \(\Coh^{\ge 2}(X)\) denote the full subcategory of \(\Coh(X)\) consisting of sheaves with support of codimension at least \(2\), and
    \[\Coh_{(1)}(X) = \Coh(X) / \Coh^{\ge 2}(X)\]
    is the quotient abelian category. For a fixed phase \(\phi\), we let \(\Coh_{(1)}^\phi(X)\) denote the full subcategory of \(\Coh_{(1)}(X)\) of semistable objects of phase \(\phi\). In the setting of \Cref{thm:ym}, the assumption that \(\gr(\mcE)\) is locally free implies that the lattice of subobjects in \(\Coh_{(1)}^\phi(X)\) coincides with the lattice of subbundles of the same slope of \(\mcE\).
\end{remark}

We state a collection of results, which we will prove in the subsequent subsections, which together imply \Cref{thm:ym}. First, we have an infinite-dimensional analogue of \Cref{prop:move-orbit}. 

\begin{proposition}
    \label{prop:kuranishi-slice-move-orbit}
    There exists \(g \in \mcG\) such that \(\delbar_0 \in \overline{G \cdot (g \cdot \delbar)}\).
\end{proposition}

Now let \(h(t) \in \mcG/\mcK\) denote the Yang-Mills flow starting from \(g \cdot \delbar\).

\begin{definition}
    The \emph{reduced Yang-Mills flow} is the corresponding flow \(\widehat h(t)\) in \(G/K\).
\end{definition}

We also have an infinite-dimensional analogue of \Cref{prop:bounded-stabiliser}.

\begin{proposition}
    \label{prop:kuranishi-slice-reduction}
    After embedding \(G/K\) into \(\mcG/\mcK\) as a totally geodesic submanifold,
    \[d(h(t), \widehat h(t)) = O(1).\]
\end{proposition}

We have now reduced from the setting of an infinite-dimensional group acting on an infinite-dimensional manifold, to a finite-dimensional group acting on an infinite-dimensional manifold. The next step is to reduce to the case of a finite-dimensional manifold.

\begin{theorem}
    [Kuranishi slice] 
    \label{thm:kuranishi-slice-ym}
    There exists a neighbourhood \(B\) of \(0\) in \(H^1\) and a \(K\)-equivariant holomorphic embedding \(\Phi \colon B \to \mcD\), such that
    \begin{enumerate}[(i)]
        \item \(\Phi(0) = \delbar_0\),
        \item \(\dd\Phi_0 = \id\),
        \item For \(g \in G\) and \(b \in B\), if \(g \cdot b \in B\), then \(\Phi(g \cdot b) = g \cdot \Phi(b)\),
        \item If \(\Phi(b)\) is integrable, and \(\dd\Phi_b(u)\) is tangent to the \(\mcG\)-orbit through \(\Phi(b)\), then \(u\) is tangent to the \(G\)-orbit through \(b\),
        \item for any sufficiently small deformation \(\delbar\) of \(\delbar_0\), there exists \(g \in \mcG\) such that \(g \cdot \delbar \in \Phi(B)\).
    \end{enumerate}
\end{theorem}

\begin{remark}
    The original theorem of Kuranishi \cite{kuranishiNewProofExistence1965} is for deformations of complex manifolds, though the same proof applies for vector bundles, and proves (i), (ii) and (v). The above statement is the vector bundle analogue of \cite[Lemma 6.1]{chenCalabiFlowGeodesic2014}, which is for cscK metrics, and the proof for cscK metrics can be adapted to the Hermitian-Yang-Mills setting. Related results can also be found in \cite{buchdahlPolystableBundlesRepresentations2022}.
\end{remark}

Now as the Kuranishi slice is a locally \(G\)-equivariant holomorphic embedding, we can find \(v \in B\) such that \(\Phi(v) = g \cdot \delbar\). Thus, after pulling back the K\"ahler metric and moment map to \(B\), we may view the Kuranishi slice as a finite-dimensional K\"ahler manifold with a \(G\)-action, and the reduced Yang-Mills flow corresponds to the moment map flow on \(B\). We are now in the setting where we can apply \Cref{thm:vector-space}. To relate the flow to the HKKP filtration of the lattice of subbundles of \(\mcE\) with the same slope, we have the following result.

\begin{proposition}
    \label{prop:kuranishi-slice-quiver}

    There exists a quiver \(Q\) and a dimension vector \(d\), such that we can identify \(H^1 = \Rep(Q, d)\) and \(G = G(d)\), such that the \(G\)-action on \(H^1\) is the same as the \(G(d)\)-action on \(\Rep(Q, d)\). Moreover, for \(v \in B\), we can identify the lattices
    \begin{enumerate}[(i)]
        \item of subrepresentations of \(v\) in \(\Rep(Q, d)\),
        \item of subbundles of the same slope of \((E, \Phi(v))\).
    \end{enumerate}
\end{proposition}

The proof of \Cref{thm:ym} follows from the above propositions, together with \Cref{cor:hkkp-quiver}. In the subsequent subsections, we prove the above propositions.

\subsection{Infinite-dimensional normal forms}

\label{subsec:ym-normal-form}

In this subsection, we prove \Cref{prop:kuranishi-slice-move-orbit} and \Cref{prop:kuranishi-slice-reduction}, following the arguments in \cite{chenCalabiFlowGeodesic2014} for the Calabi flow. There are however some simplifications passing from the Calabi flow to the Yang-Mills flow. 

Throughout, we will appeal to the Nash-Moser implicit function theorem. We will now briefly review tame Fr\'echet spaces and the Nash-Moser implicit function theorem, and refer to \cite{hamiltonInverseFunctionTheorem1982} for details. Let \(V\) be a Fr\'echet space. A sequence \(\abs{\cdot}_n\) of seminorms on \(V\) is called a \emph{grading} if
\[\abs{v}_0 \le \abs{v}_1 \le \cdots\]
and if the seminorms define the topology. One example is given by \(\abs{\cdot}_k = \abs{\cdot}_{C^k}\). 

A linear map \(L \colon V \to W\) between two graded Fr\'echet spaces is \emph{tame} if there exists \(r, n_0 \in \Z_{\ge 0}\) and \(C > 0\) such that
\[\abs{L(v)}_n \le C\abs{v}_{n + r}\]
for all \(n \ge n_0\).

Let \(B\) be a Banach space, and let \(\Sigma(B)\) denote the space of sequences \((x_k)\) in \(B\), with
\[\abs{(x_k)}_n = \sum_{k=0}^\infty e^{nk}\norm{x_k} < \infty.\]
The seminorms \(\abs{\cdot}_n\) define a grading on \(\Sigma(B)\). A graded Fr\'echet space \(V\) is \emph{tame} if there exists a Banach space \(B\), and tame linear maps \(L \colon V \to \Sigma(B), M \colon \Sigma(B) \to V\), such that \(M \circ L = \id_V\).

Finally, let \(V, W\) be graded Fr\'echet spaces, \(U \subseteq V\) open, \(F \colon U \to W\) a map. We say that \(F\) is \emph{tame} if it is continuous, and if there exists \(r, n_0 \in \Z_{\ge 0}\) and \(C > 0\) such that
\[\abs{F(v)}_n \le C(1 + \abs{v}_{n + r})\]
for all \(v \in U\) and \(n \ge n_0\). We say that \(F\) is \emph{smooth tame} if it is smooth, and all of its derivatives are tame. A \emph{tame manifold} is a space locally modelled on a tame Fr\'echet space, with smooth tame transition maps.

\begin{theorem}
    [Nash-Moser implicit function theorem, {\cite[Theorem III.1.1.1]{hamiltonInverseFunctionTheorem1982}}] Let \(V, W\) be tame Fr\'echet spaces, \(U \subseteq V\) open, \(F \colon U \to W\) a smooth tame map. Suppose the equation \(D F(v)x = w\) has a unique solution \(x = Q(v)w\) for all \(v \in U\), \(x \in V\) and \(w \in W\). Moreover, suppose \(Q \colon U \times W \to V\) is a smooth tame map. Then \(F\) is locally invertible, and each local inverse is a smooth tame map.
\end{theorem}

First, we let \(N\) denote the orthogonal complement of the image of \(\delbar_0 \colon A^0(X, \End(E)) \to A^{0, 1}(X, \End(E))\). Let \(\mfk = \Lie(\mcK)\) and \(\mfk_0 = \Lie(K)\). Then we have an \(L^2\)-orthogonal decomposition
\[\mfk = \mfk_0 \oplus \mfm.\]
The model space is
\[Y = \mcK \times_K (\mfm \oplus N).\]

We note that \(\mcK\) is a smooth tame Lie group, and the action of \(K\) on \(\mcK \times (\mfm \oplus N)\), given by
\[k \cdot (g, \rho, v) = (gk^{-1}, k \cdot \rho, k \cdot v)\]
is smooth tame and free, and so the quotient \(Y\) is a smooth tame manifold.

The \(K\)-action on \(N\) is linear, and so we have a moment map given by
\begin{align*}
    \mu_N \colon N &\to \mfk_0 \\
    \langle \mu_N(v), \xi\rangle &= \frac{1}{2}\Omega(\xi \cdot v, v).
\end{align*}
Since \(\mfk_0\) is finite dimensional, this is a well-defined tame map. As in the finite-dimensional case, we have a canonically defined symplectic form on \(Y\), given by
\begin{align*}
    &\Omega_Y\vert_{[g, \rho, v]}((L_g\xi_1, \rho_1, v_1), (L_g\xi_2, \rho_2, v_2)) \\
    &\quad = \langle \rho_2 + \dd_v\mu_N(v_2), \xi_1\rangle - \langle \rho_1 + \dd_v\mu_N(v_1), \xi_2 \rangle \\
    &\quad\quad + \langle \rho + \mu_N(v), [\xi_1, \xi_2]\rangle + \Omega(v_1, v_2), \\
    &\quad = \langle \xi_1, \rho_2\rangle \\
    &\quad\quad -\langle \rho_1 + \dd_v\mu_N(v_1) - [\rho + \mu_N(v), \xi_1], \xi_2\rangle \\
    &\quad\quad + \Omega_0(v_1, v_2) + \langle \xi_1, \dd_v\mu_N(v_2)\rangle.
\end{align*}

The left \(\mcK\)-action on \(Y\) is Hamiltonian, and has moment map
\begin{align*}
    \mu_Y \colon Y &\to \mfk, \\
    \mu_Y[g, \rho, v] &= \Ad(g)(\rho + \mu_N(v)).
\end{align*}

Next, we define a map \(\Theta \colon \mfm \to i\mfk \cdot\delbar_0\) by
\[\Theta(\rho) = -i\delbar_0(\Delta^{-1}\rho),\]
where \(\Delta = \delbar_0^*\delbar_0 + \delbar_0\delbar_0^*\) is the Laplacian. We define
\begin{align*}
    \Phi \colon Y &\to \mcD, \\
    [g, \rho, v] &\mapsto g \cdot (\delbar_0 + \Theta(\rho) + v).
\end{align*}
This can be rewritten as
\[\Phi[g, \rho, v] = \delbar_0 + g\delbar_0(g^{-1}) + g(\Theta(\rho) + v)g^{-1}.\]

\begin{lemma}
    \(\Phi\) is smooth tame, with a local smooth tame inverse around \([\id, 0, 0]\).
\end{lemma}

\begin{proof}
    By definition, \(\Phi\) is smooth and tame. To be able to apply the inverse function theorem, we need to study the derivative of \(\Phi\) near \([\id, 0, 0]\). At \(\delta = [g, \rho, v]\),
    \begin{align*}
        D_\delta\Phi \colon \mfm \oplus \mfm \oplus N &\to A^{0, 1}(X, \End(E)) \\
        (\xi, \psi, u)&\mapsto g(-\delbar_0\xi + [\xi, \Theta(\rho) + v] + \Theta(\psi) + u)g^{-1}.
    \end{align*}
    To show invertibility, it suffices to consider the case when \(g = \id\). Set \(\delbar_\delta = \Phi(\delta) = \delbar_0 + \Theta(\rho) + v\), and so
    \[D_\delta\Phi(\xi, \psi, u) = -\delbar_\delta\xi + \Theta(\psi) + u.\]

    Define \(L_\delta \colon \mfm \to \mfm\) by
    \[L_\delta(\xi) = \Re\delbar_0^*\delbar_\delta(\xi) = \Delta\xi + \Re\delbar_0^*[\Theta(\rho) + v, \xi],\]
    where
    \[\Re(\xi) = \frac{\xi - \xi^*}{2}\]
    is the skew-hermitian part of \(\xi\). When \(\delta = [\id, 0, 0]\), \(L_\delta = \Delta\), and so \(L_\delta\) is invertible for \(\delta\) sufficiently small.

    Suppose \(D_\delta\Phi(\xi, \psi, u) = \nu\). Then
    \[L_\delta\xi = -\Re\delbar_0^*(D_\delta\Phi(\xi, \psi, u)) = -\Re\delbar_0^*\nu,\]
    and so we must have that \(\xi = -L_\delta^{-1}\Re\delbar_0^*\nu\). We claim that this choice of \(\xi\) defines the inverse to \(D_\delta\Phi\).

    By construction, we have a decomposition
    \[A^{0, 1}(X, \End(E)) = \im(\delbar_0\vert_{\mfm}) \oplus \im(\delbar_0\vert_{i\mfm}) \oplus N.\]
    If we set \(\xi = -L_\delta^{-1}\Re\delbar_0^*\nu\), and \(\eta = \nu + \delbar_\delta\xi\), then
    \[\Re\delbar_0^*\eta = \Re\delbar_0^*\nu - \Re\delbar_0^*\delbar_\delta L_\delta^{-1}\Re\delbar_0^*\nu = \Re\delbar_0^*\nu - L_\delta L_\delta^{-1}\Re\delbar_0^*\nu = 0.\]
    Thus, \(\eta \in \ker(\Re\delbar_0^*) = \im(\delbar_0\vert_{i\mfm}) \oplus N\). We may then write \(\eta = \delbar_0\psi + u\), for some \(\psi \in \mfm\) and \(u \in N\), by elliptic theory for \(\Delta\). Observe that the inverse map is smooth tame, and so we can apply the Nash-Moser inverse function theorem \cite{hamiltonInverseFunctionTheorem1982}.
\end{proof}

By the above lemma, we may pullback the symplectic form and complex structure from \(\mcD\) to a neighbourhood \(U\) of \([\id, 0, 0]\) in \(Y\). Let us write \(\Omega' = \Phi^*\Omega\) and \(J' = \Phi^*J\). On the other hand, there is also a canonical almost complex structure \(J_0\) on \(U\), given by
\begin{align*}
    J_0 \colon \mfm \oplus \mfm \oplus N &\to \mfm \oplus \mfm \oplus N, \\
    (\xi, \rho, v) &\mapsto (-\Delta^{-1}\rho, \Delta\xi, iv).
\end{align*}
Observe that \(J' = J_0\) at \([\id, 0, 0]\).

\begin{proposition}
    There exist neighbourhoods \(U_1, U_2, V_1, V_2\) of \([\id, 0, 0]\) in \(Y\), with \(U_2 \subseteq U_1\), and \(\mcK\)-equivariant smooth tame maps
    \begin{align*}
        \Sigma_1 \colon U_1 &\to V_1, \\
        \Sigma_2 \colon V_2 &\to U_2
    \end{align*}
    fixing the \(\mcK\)-orbit of \([\id, 0, 0]\), such that \(\Sigma_1 \circ \Sigma_2 = \id\), and
    \begin{align*}
        \Sigma_1^*\Omega_Y &= \Omega', \\
        \Sigma_2^*\Omega' &= \Omega_Y.
    \end{align*}
    Moreover, there exist \(k, s\) such that for any \(X \in N\), and \([g, \rho, v] \in V_2\),
    \[\abs{(D\Sigma_1)\circ J' \circ (D\Sigma_2)(X) - J_0(X)}_k \le C_{k, s}\abs{X}_{k + s}r_{k+s}^2,\]
    and at \([\id, 0, 0]\),
    \[(D\Sigma_1)\circ J' \circ (D\Sigma_2) = J_0.\]

    Here, the estimate holds in the tame sense, \(r_{k+s}\) is the distance between \([g, \rho, v]\) and \([\id, 0, 0]\) in the \(C^{k+s}\)-norm.
\end{proposition}

\begin{proof}
    The idea of the proof is the same as the finite-dimensional case, which relies on Moser's trick. The main difficulty is that we now need to solve an ordinary differential equation in infinite dimensions.

    Let \(\alpha = \Omega' - \Omega_Y\). Define an isotopy \(f_t[g, \rho, v] = [g, t\rho, tv]\) of \(Y\), which is given by the vector field \(X_t[g, \rho, v] = [0, \rho, v]\). Using this, we may write \(\alpha = \dd\theta\), where
    \[\theta = \int_0^1f_t^*(\iota_X\alpha)\dd t.\]

    Next, for the Moser trick, we set \(\Omega_t = (1 - t)\Omega_Y + t\Omega'\), and we seek a vector field \(Y_t\) such that 
    \[\iota_{Y_t}\Omega_t = \theta.\] 
    In finite dimensions, non-degeneracy of \(\Omega_t\) implies the existence of such a vector field, however in infinite dimensions this is now an equation we need to solve. The left hand side is given by
    \[(1 - t)\iota_{Y_t}\Omega_Y + t\iota_{Y_t}\Omega',\]
    and the right hand side is given by
    \[\theta_{[g, \rho, v]}(Z) = \int_0^1(\Omega' - \Omega_Y)_{[g, s\rho, sv]}((0, \rho, v), f_{s*}Z) \dd s.\]
    We now compute all of the terms individually.

    Note that \(Y_t\) is \(\mcK\)-invariant, and so we may assume \(g = \id\). Recall that
    \begin{align*}
        &\Omega_Y\vert_{[\id, \Delta\rho, v]}((\xi_1, \Delta\rho_1, v_1), (\xi_2, \Delta\rho_2, v_2)) \\
        &\quad = \langle \Delta\xi_1, \rho_2\rangle_{L^2} \\
        &\quad\quad -\langle \Delta\rho_1 + \dd_v\mu_N(v_1) - [\Delta\rho + \mu_N(v), \xi_1], \xi_2\rangle_{L^2} \\
        &\quad\quad + \Omega(v_1, v_2) + \langle \xi_1, \dd_v\mu_N(v_2)\rangle_{L^2}.
    \end{align*}
    The change of variable to \(\Delta\rho\) is to simplify the formula for \(\Omega'\). We can rewrite the last term as
    \[\langle\xi_1, \dd_v\mu_N(v_2)\rangle_{L^2} = \langle (\dd_v\mu_N)^*(\xi_1), v_2\rangle_{L^2} = -\Omega(i(\dd_v\mu_N)^*(\xi_1), v_2).\]
    
    Next, for \(\delta = [\id, \Delta\rho, v]\), we have
    \begin{align*}
        &\Omega'\vert_{[\id, \Delta\rho, v]}((\xi_1, \Delta\rho_1, v_1), (\xi_2, \Delta\rho_2, v_2)) \\
        &\quad = \Omega(-\delbar_\delta\xi_1 - i\delbar_0\rho_1 + v_1, -\delbar_\delta\xi_2 - i\delbar_0\rho_2 + v_2), \\
        &\quad = -\Im\langle-\delbar_\delta\xi_1 - i\delbar_0\rho_1 + v_1, -\delbar_\delta\xi_2 - i\delbar_0\rho_2 + v_2\rangle_{L^2}, \\
        &\quad = -\Im\langle \delbar_\delta^*\delbar_\delta\xi_1 + i\delbar_\delta^*\delbar_0\rho_1 - \delbar_\delta^*v_1, \xi_2\rangle_{L^2} \\
        &\quad\quad -\Im\langle -i\delbar_0^*\delbar_\delta\xi_1 + \delbar_0^*\delbar_0\rho_1 + i\delbar_0^*v_1, \rho_2\rangle_{L^2} \\
        &\quad\quad -\Im\langle -\delbar_\delta\xi_1 - i\delbar_0\rho_1 + v_1, v_2\rangle_{L^2}.
                                            \end{align*}
    Let us set
    \[\Im(\xi) = \frac{\xi + \xi^*}{2i},\]
    to be \(-i\) times the hermitian part of \(\xi\), so that for any \(\xi \in A^0(X, \End(E))\), we have a decomposition
    \[\xi = \Re(\xi) + i\Im(\xi).\]

    This gives
    \begin{align*}
        &\Omega'\vert_{[\id, \Delta\rho, v]}((\xi_1, \Delta\rho_1, v_1), (\xi_2, \Delta\rho_2, v_2)) \\
        &\quad = \langle -\Im \delbar_\delta^*\delbar_\delta\xi_1 - \Re \delbar_\delta^*\delbar_0\rho_1 + \Im \delbar_\delta^*v_1, \xi_2\rangle_{L^2} \\
        &\quad\quad + \langle \Re\delbar_0^*\delbar_\delta\xi_1, \rho_2\rangle_{L^2} \\
        &\quad\quad + \Omega(-\delbar_\delta\xi_1 - i\delbar_0\rho_1 + v_1, v_2).
    \end{align*}
    Note that we used the fact that \(\Im(\delbar_0^*\delbar_0\rho_1) = 0\), and that \(\delbar_0^*v_1 = 0\).

    We now set \(Y_t = (\xi_1(t), \Delta\rho_1(t), v_1(t))\) and \(Z = (\xi_2, \Delta\rho_2, v_2)\). Set \(\delta_s = [\id, s\Delta\rho, sv]\) and \(\delbar_{\delta_s} = \Phi(\delta_s)\). Then
    \[\Omega_Y\vert_{\delta_s}((0, \Delta\rho, v), f_{s*}Z) = -\langle\Delta\rho + \dd_{sv}\mu_N(v), s\xi_2\rangle_{L^2} + \Omega(v, sv_2).\]
    Similarly,
    \begin{align*}
        \Omega'\vert_{\delta_s}((0, \Delta\rho, v), f_{s*}Z) &= \langle -\Re\delbar_{\delta_s}^*\delbar_0\rho + \Im\delbar_{\delta_s}^*v, s\xi_2\rangle_{L^2} \\
                        &\quad + \Omega(-i\delbar_0\rho + v, sv_2).
    \end{align*}

    Thus,
    \begin{align*}
        \theta\vert_\delta(Z) &= \int_0^1(\Omega' - \Omega_Y)\vert_{\delta_s}((0, \Delta\rho, v), f_{s*}Z) \dd s \\
        &= \int_0^1 s\langle -\Re\delbar_{\delta_s}^*\delbar_0\rho + \Im\delbar_{\delta_s}^*v + \Delta\rho + \dd_{sv}\mu_N(v), \xi_2\rangle_{L^2} \dd s, \\
                &\quad + \int_0^1s\Omega(-i\delbar_0\rho, v_2)\dd s.
    \end{align*}
    Each component of the equation \(\iota_{Y_t}\Omega_t = \theta\) yields an equation. More precisely, we obtain
    \begin{equation}
        \label{eq:moser-1}
        \begin{split}
            &(1 - t)(-\Delta\rho_1 - \dd_v\mu_N(v_1) + [\Delta\rho + \mu_N(v), \xi_1]) \\
            &\quad + t(-\Im\delbar_\delta^*\delbar_\delta\xi_1 - \Re\delbar_\delta^*\delbar_0 \rho_1 + \Im\delbar_\delta^*v_1) \\
            &= \int_0^1 s(-\Re \delbar_{\delta_s}^*\delbar_0\rho + \Im \delbar_{\delta_s}^*v + \Delta\rho + \dd_{sv}\mu_N(v)) \dd s \mod \mfk_0,
        \end{split}
    \end{equation}
    \begin{equation}
        \label{eq:moser-2}
            (1 - t)\Delta\xi_1 + t\Re \delbar_0^*\delbar_\delta\xi_1 = 0 \mod \mfk_0,
    \end{equation}
    and
    \begin{equation}
        \label{eq:moser-3}
            (1 - t)(v_1 - i(\dd_v\mu_N)^*(\xi_1)) + t(-\delbar_\delta\xi_1 - i\delbar_0\rho_1 + v_1) = \int_0^1si\delbar_0\rho\dd s \mod \im(\delbar_0).
    \end{equation}

    We claim that the above system of equations admits a unique smooth solution when \([\rho, v]\) is sufficiently close to \([0, 0]\) in \(C^\infty\). To see this, notice when \([\rho, v] = [0, 0]\), it reduces to the system
    \begin{align*}
        \Delta\rho_1 &= \alpha, \\
        \Delta\xi_1 &= \beta, \\
        v_1 &= \gamma,
    \end{align*}
    which has a unique solution by elliptic theory. For \([\rho, v]\) small, unique solvability follows from the implicit function theorem for Banach spaces, and again by ellipticity, the solutions are smooth and with (tame) estimates.

    Next, we prove that there are two neighbourhoods \(N_1, N_2\) of \(0\) in \(\mfm \oplus N\), and a smooth tame map \(F \colon N_1 \to C^\infty([0, 1], \mfm \oplus N)\), such that the time 1 evaluation \(F_1\) defines a smooth tame map \(N_1 \to N_2\), and for any \([\rho, v] \in N_1\),
    \[\begin{cases}
        \dv{t}F_t(\rho, v) &= (\rho_1(t), v_1(t)) \\
        F_0(\rho, v) &= (\rho, v).
    \end{cases}\]
    To prove this claim, we use the Nash-Moser implicit function theorem. Define a map
    \begin{align*}
        H \colon C^\infty([0, 1], \mfm \oplus N) &\to (\mfm \oplus N) \times C^\infty([0, 1], \mfm \oplus N) \\
        (\rho(t), v(t)) &\mapsto (\rho(0), v(0)) \times (\dot\rho(t) - \rho_1(t), \dot v(t) - v_1(t)).
    \end{align*}
    The functions \(\rho_1(t), v_1(t)\) are defined implicitly by the system of \cref{eq:moser-1,eq:moser-2,eq:moser-3} from \(\rho, v\), and so \(H\) is a smooth tame map, and \(H(0) = 0\). We need to show that for \(x = (\rho(t), v(t))\) close to zero, the derivative of \(H\) at \(x\) is invertible, with smooth tame inverse. Let \(\delta x = (\widetilde\rho, \widetilde v)\). Then the derivative of \(H\) at \(x\) along \(\delta x\) is given by \((\widetilde \rho(0), \widetilde v(0)) \times (\dot{\widetilde\rho}(t) - \delta\rho_1(\widetilde \rho, \widetilde v), \dot{\widetilde v}(t) - \delta v_1(\widetilde \rho, \widetilde v))\). Thus, the invertibility of \(D_xH\) is equivalent to the solvability of the Cauchy problem of the following linear system along \((\rho(t), v(t))\):
    \[\begin{cases}
        \dv{t}(\alpha, u) &= (\delta\rho_1(\alpha, u), \delta v_1(\alpha, u)) + (\beta, q),\\
        (\alpha(0), u(0)) &= (\widetilde\rho(0), \widetilde v(0)).
    \end{cases}\]
    Thus, we need to linearise \cref{eq:moser-1,eq:moser-3}. By collecting highest order terms, we obtain the following.
    \[\begin{cases}
        \Delta\dot\alpha(t) &= A_2(\rho(t), v(t))\alpha(t) + B_1(\rho(t), v(t))u(t) + \beta(t), \\
        \dot u(t) &= C_1(\rho(t), v(t))\dot\alpha(t) + D_1(\rho(t), v(t))\alpha(t) + E_0(\rho(t), v(t))u(t) + q(t), \\
        (\alpha(0), u(0)) &= (\widetilde\rho(0), \widetilde v(0)),
    \end{cases}\]
    where \(A_2, B_1, C_1, D_1, E_0\) are differential operators, with order indicated by the subscript, depending smoothly on \((\rho(t), v(t))\). Let
    \[w(t) = u(t) - C_1(\rho(t), v(t))\alpha(t).\]
    Then the above system can be rewritten as
    \[\begin{cases}
        \dot\alpha(t) &= \Delta^{-1}\widehat A_2(\rho(t), v(t))\alpha(t) + \Delta^{-1}\widehat B_1(\rho(t), v(t))w(t) + \beta(t), \\
        \dot w(t) &= \widehat D_1(\rho(t), v(t))\alpha(t) + \widehat E_0(\rho(t), v(t))w(t) + \widehat q(t), \\
        (\alpha(0), w(0)) &= (\widetilde\rho(0), \widetilde v(0) - C_1\widetilde\rho(0)).
    \end{cases}\]
    Let us now pass to a Sobolev completion. For all \(k\) sufficiently large, we have an \(L^2\)-orthogonal decomposition
    \[L^2_k(X, \End_{SH}(E, h)) = \mfk_0 \oplus \mfm_k,\]
    as well as
    \[L^2_{k-1}(X, \Lambda^{0, 1} \otimes \End(E)) = \im(\delbar_0) \oplus N_{k-1}.\]
    We note that \(\Delta\) defines an invertible operator \(\mfm_k \to \mfm_{k-2}\). The spaces \(\mfm_k\) and \(N_{k-1}\) can be written as kernels of continuous linear operators, and hence are Banach. 
    
    Next, we define an operator
    \begin{align*}
        P_k \colon \mfm_k \oplus N_{k-1} &\to \mfm_k \oplus N_{k-1} \\
        (\alpha, w) &\mapsto (\Delta^{-1}\widehat A_2\alpha + \Delta^{-1}\widehat B_1w, \widehat D_1\alpha + \widehat E_0w).
    \end{align*}
    This defines a continuous linear operator which depends smoothly on \(t\), and so we have a bound on \(\norm{P_k}\).
    
    The Cauchy problem is the system
    \[\begin{cases}
        \dv{t}(\alpha, w) &= P_k(\alpha, w) + (\beta, \widehat q), \\
        (\alpha(0), w(0)) &= (\widetilde\rho(0), \widetilde v(0) - C_1\widetilde\rho(0)).
    \end{cases}\]
    By the above discussion, a unique solution exists by Picard iteration.
    
    By uniqueness, the solution is independent of \(k\), and so is smooth. Moreover, the solution depends tamely on \((\rho(t), v(t)), (\beta(t), q(t))\) and the initial condition. Thus, by the Nash-Moser implicit function theorem \cite{hamiltonInverseFunctionTheorem1982}, \(H\) has a local smooth tame inverse. We let \(F = H^{-1}(\cdot, 0)\).

    Now for \([g, \rho, v]\) close to \([\id, 0, 0]\), we obtain a path \((\rho(t), v(t)) = F_t(\rho, v)\). Then we can solve the ordinary differential equation
    \[\dot g(t) = L_{g(t)}\xi_1(t),\]
    where \(\xi_1(t)\) is determined by \((\rho(t), v(t))\). Then by \(\mcK\)-invariance, \((g(t), \rho(t), v(t))\) is an integral curve of \(Y_t\). Now define
    \begin{align*}
        \Sigma_2 \colon V_2 &\to U_2 \\
        [g, \rho, v] &\mapsto [g(1), F_1(\rho, v)].
    \end{align*}
    Then by the previous arguments, we know that \(\Sigma_2\) is smooth tame and fixes \(\mcK \cdot [\id, 0, 0]\).

    It follows from \cref{eq:moser-1,eq:moser-2,eq:moser-3} that we have a tame estimate
    \[\abs{v_1(t)}_k \le C_k \left(\abs{\rho(t)}_{k + s_1} + \abs{v(t)}_{k + s_1}\right)^3.\]
    Now \(\abs{(\rho(t), v(t))}_{k + s_1} \le C_k \abs{(\rho(0), v(0))}_{k + s}\) for some \(s > s_1\), and so
    \[\abs{v(1) - v(0)}_k \le C\left(\abs{\rho(0)}_{k+s} + \abs{v(0)}_{k+s}\right)^3.\]

    By symmetry, one can obtain the map \(\Sigma_1\), and check that the required properties hold.
\end{proof}

We are now in a position to prove \Cref{prop:kuranishi-slice-reduction}.

\begin{proof}
    [Proof of \Cref{prop:kuranishi-slice-reduction}]
    We use the notation as in the statement. Let \(\alpha(t) \in \mcD\) denote the Yang-Mills flow, and let \(\widetilde\alpha(t) \in B \subseteq H^1\) denote the flow in the Kuranishi slice. Set \(\widehat\alpha(t) = \Phi(\widetilde\alpha(t))\).
    
    Using the above result, we may work instead with \((V_2, \Omega_Y, \widehat J)\), where \(\widehat J - J_0 = O(r^2)\), in the tame sense. Let \(\psi_t(s)\), \(s \in [0, 1]\), denote the geodesic in \(\mcG/\mcK\) connecting \(h(t)\) to \(\widehat h(t)\), and let \(L(t)\) denote the length of \(\psi_t\). Using convexity of the Donaldson functional, as in the proof of \Cref{prop:flow-one-param-limit}, we obtain that
    \[\dv{t}L(t) \le \abs{\mu(\widehat \alpha(t)) - \widehat\mu(\widehat \alpha(t))}_{L^2} \lesssim \abs{\mu(\widehat\alpha(t)) - \widehat\mu(\widehat\alpha(t))}_0,\]
    where we used H\"older's inequality for the second inequality.
    
    Let us write \(\Phi(\widehat\alpha(t)) = [g(t), \rho(t), v(t)]\). The curves \(g(t)\) and \(\rho(t)\) are uniquely determined by the value at \(t = 0\), if we assume that \(g(t)^{-1}\dot g(t) \in \mfm\). Moreover, by the finite-dimensional result, \(\widetilde\alpha(t) \to 0\), with \(\abs{\widetilde\alpha(t)} = O(t^{-1/2})\).
    
    Then
    \[\mu(\widehat \alpha(t)) - \widehat\mu(\widehat \alpha(t)) = g(t)\rho(t)g(t)^{-1}.\]
    Let \(\widehat f = \abs{\widehat\mu}^2\). Then
    \begin{align*}
        \grad \widehat f &= \widehat J[0, \ad_{\mu_N(v)}\rho, \mu_N(v) \cdot v] \\
        &= [-\Delta^{-1}\ad_{\mu_N(v)}\rho, 0, i\mu_N(v) \cdot v] + (\widehat J - J_0)[0, \ad_{\mu_N(v)}\rho, \mu_N(v) \cdot v].
    \end{align*}
    By considering the second factor only, we obtain that
    
    \begin{align*}
        \abs{\dot{\widehat \alpha}(t)}_k &= \abs{\grad \widehat f(\widehat \alpha(t))}_k \\
        &\le C\left(d_{k+s}(\widehat \alpha(t), \delbar_0)\abs{\mu_N(v(t))}_{k+s}\abs{\rho(t)}_{k+s} + d_{k+s}(\widehat\alpha(t), \delbar_0)^2\abs{\mu_N(v(t)) \cdot v(t)}_{k+s}\right).
    \end{align*}
    
    From the finite-dimensional setting, see for instance the results in \Cref{sec:finite-dimensional} or \cite{chenCalabiFlowGeodesic2014}, \(d_l(\widehat \alpha(t), \delbar_0), \abs{\rho(t)}_l, \abs{v(t)}_l = O(t^{-1/2})\). It then follows that
    \[\abs{\dot{\widehat \alpha}(t)}_k \le C\left(t^{-3/2}\abs{\rho(t)}_{k+s} + t^{-5/2}\right).\]
    Since \(\lim_{t \to \infty}\rho(t) = 0\), it follows that \(\abs{\rho(t)} \le Ct^{-3/2}\). To see this, \(\abs{\dot\rho(t)}_k \le \abs{\dot{\widehat\alpha}(t)}_k\), and so
    \[\abs{\dot\rho(t)}_k \le C\left(t^{-3/2}\abs{\rho(t)}_{k+s} + t^{-5/2}\right) \le Ct^{-2}.\]
    Using the bound \(\abs{\rho(t)}_{k+s} \le Ct^{-1/2}\) and the mean value inequality, we see that \(\abs{\rho(t)}_k \le Ct^{-1}\). Iterating, we obtain \(\abs{\rho(t)}_k \le Ct^{-3/2}\). Thus, \(L(t)\) is bounded.
\end{proof}

Finally, we prove \Cref{prop:kuranishi-slice-move-orbit}.

\begin{proof}
    [Proof of \Cref{prop:kuranishi-slice-move-orbit}]

    Let \(\delbar\) be an integrable Dolbeault operator near \(\delbar_0\), such that the Yang-Mills flow starting at \(\delbar\) converges to \(\delbar_0\). Let \(\alpha(t)\) denote the Yang-Mills flow. By property (v) of \Cref{thm:kuranishi-slice-ym}, we may perturb \(\alpha(t)\) to \(\overline\alpha(t)\) within its \(\mcG\)-orbit, such that \(\overline\alpha(t) = \Phi(b(t))\). Since \(\dot{\overline\alpha}(t)\) is tangent to the \(\mcG\)-orbit, it follows that \(\dot b(t)\) is tangent to the \(G\)-orbit, by (iv) of \Cref{thm:kuranishi-slice-ym}. By assumption, \(b(t) \to 0\), and \(b(t)\) lies in a fixed \(G\)-orbit. Thus, \(0 \in \overline{G \cdot b(0)}\), and the result follows by the local \(G\)-equivariance of \(\Phi\).
\end{proof}

\subsection{Quiver representations}

In this subsection, we prove \Cref{prop:kuranishi-slice-quiver}. Roughly, this follows \cite{haidenSemistabilityModularLattices2023,haidenIteratedLogarithmsGradient2018}, where we view \(H^{0, 1}(X, \End(\mcE_0))\) as a \(H^0(X, \End(\mcE_0))\)-bimodule. Decomposing the bimodule into simple pieces defines a quiver. We refer to \cite[56]{haidenSemistabilityModularLattices2023} and \cite[Lemma 7.4.3]{ibaneznunezRefinedHarderNarasimhanFiltrations2024} for a statement in terms of bimodules, where in \cite[Theorem 7.4.10]{ibaneznunezRefinedHarderNarasimhanFiltrations2024} it is applied to study the moduli stack of objects in an abelian category \cite{artinAbstractHilbertSchemes2001,alperExistenceModuliSpaces2023}. Similar ideas have also been considered in \cite{kirwanModuliSpacesBundles2004}, though not using the language of quivers.

The main results are \Cref{lem:kuranishi-slice-compatibility,lem:kuranishi-slice-compatibility-2}, which allow us to relate subrepresentations of the quiver representation to subbundles of the deformation.

Let \((X, \omega)\) be compact K\"ahler, \(E \to X\) a smooth complex vector bundle, and \(\delbar_0\) an integrable Dolbeault operator on \(E\), such that \((E, h, \delbar_0)\) is Hermitian-Yang-Mills. Let us write \(\mcE_0 = (E, \delbar_0)\), and we let \(\mu = \mu(E)\) denote the slope of \(E\). Since \(\mcE_0\) is slope polystable,
\[\mcE_0 = \bigoplus_{i=1}^s \mcF_i \otimes V_i,\]
where \(\mcF_i\) are pairwise non-isomorphic stable bundles of slope \(\mu\), and \(V_i\) is a finite-dimensional complex vector space. Let \(F_i\) denote the underlying smooth vector bundle of \(\mcF_i\), so that \(\mcF_i = (F_i, \delbar_0\vert_{F_i})\). In particular, 
\[\End(\mcE_0) = \bigoplus_{i, j}\Hom(\mcF_i, \mcF_j) \otimes \Hom(V_i, V_j).\]
Taking cohomology, we obtain that
\begin{align*}
    H^0(X, \End(\mcE_0)) &= \bigoplus_{i=1}^s \C \id_{F_i} \otimes \End(V_i), \\
    H^{0, 1}(X, \End(\mcE_0)) &= \bigoplus_{i, j}H^{0, 1}(X, \Hom(\mcF_i, \mcF_j)) \otimes \Hom(V_i, V_j).
\end{align*}
By the above description,
\[G = \prod_i \GL(V_i),\]
and that the \(G\)-action on \(A^{0, 1}(X, \End(\mcE_0))\) is given by
\[(g_i) \cdot \left(\alpha_{ij} \otimes A_{ij}\right) = \alpha_{ij} \otimes g_j A_{ij}g_i^{-1},\]
which restricts to an action on \(H^{0, 1}(X, \End(\mcE_0))\). Thus, we define a quiver \(Q = (Q_0, Q_1)\) as follows.
\begin{enumerate}[(i)]
    \item the vertices \(Q_0\) are given by \(\{1, \dots, s\}\),
    \item between \(i\) and \(j\), we have \(\dim H^{0, 1}(X, \Hom(\mcF_i, \mcF_j))\) arrows.
\end{enumerate}
We let \(d_i = \dim(V_i)\). Choose a basis \(\alpha_{ij}^k\) of \(H^{0, 1}(X, \Hom(\mcF_i, \mcF_j))\), and isomorphisms \(V_i \cong \C^{d_i}\). We define a map
\begin{align*}
    \theta \colon H^{0, 1}(X, \End(\mcE_0)) &\to \Rep(Q, d) \\
    \sum_{i,j, k} \alpha_{ij}^k \otimes A_{ij}^k &\mapsto (A_{ij}^k).
\end{align*}
The isomorphism \(V_i \cong \C^{d_i}\) induces an isomorphism
\[G = \prod_i \GL(V_i) \cong \prod_i \GL(d_i, \C) = G(d),\]
under which \(\theta\) is \(G = G(d)\)-equivariant.

Now suppose \(\mcE = (E, \delbar_0 + \alpha)\) is a deformation of \(\mcE_0 = (E, \delbar_0)\), and that \(\mcS = (S, (\delbar_0 + \alpha)\vert_S)\) is a subbundle of \(\mcE\), with the same slope. Then \(\mcS\) and \(\mcE/\mcS\) are semistable bundles with the same slope. By combining their Jordan-H\"older filtrations, we obtain a Jordan-H\"older filtration of \(\mcE\), with \(\mcS\) being one of the terms. Thus, \(\gr(\mcS)\) is a subbundle of \(\gr(\mcE) = \mcE_0\), of the same slope, and so
\[\mcS_0 = \gr(\mcS) = (S, \delbar_0\vert_S) = \bigoplus_i \mcF_i \otimes W_i,\]
for some subspaces \(W_i \subseteq V_i\).

Suppose \(\delbar_0 + \alpha = \Phi(v)\), for some \(v \in B\). We claim that \((W_i)\) define a subrepresentation of \(\theta(v)\). More precisely, we need to show that if \(\theta(v) = (A_{ij}^k)\), then \(A_{ij}^k(W_i) \subseteq W_j\) for all \(i, j, k\).

By choosing an orthogonal splitting \(V_i = W_i \oplus U_i\), we obtain a splitting \(\mcE_0 = \mcS_0 \oplus \mcQ\), where
\[\mcQ = \bigoplus_i \mcF_i \otimes U_i.\]
We may decompose
\begin{equation}
    \label{eq:hom-decomposition}
    \Hom(V_i, V_j) = \Hom(W_i, W_j) \oplus \Hom(W_i, U_j) \oplus \Hom(U_i, W_j) \oplus \Hom(U_i, U_j).
\end{equation}
Thus, for \((A_{ij}^k) \in \Rep(Q, d)\), \(A_{ij}^k\) can be decomposed as
\[A_{ij}^k = (A_{ij}^k)^{WW} + (A_{ij}^k)^{WU} + (A_{ij}^k)^{UW} + (A_{ij}^k)^{UU},\]
and that \(A_{ij}^k(W_i) \subseteq W_j\) if and only if \((A_{ij}^k)^{WU} = 0\).

Next, for any \(\beta \in A^{0, 1}(X, \End(\mcE_0))\), we may write
\[\beta = \sum_{i,j,k} \beta_{ij}^k \otimes B_{ij}^k,\]
where \(\beta_{ij}^k \in A^{0, 1}(X, \Hom(\mcF_i, \mcF_j))\) and \(B_{ij}^k \in \Hom(V_i, V_j)\). Also by the above decomposition, we can make sense of \(B_{ij}^k(W_i) \subseteq W_j\), even in this infinite-dimensional setting.

\begin{definition}
    Given subspaces \(W = (W_i)\) of \(V = (V_i)\), we say that \(\beta \in A^{0, 1}(X, \End(\mcE_0))\) is \emph{compatible} with \(W\) if \(B_{ij}^k(W_i) \subseteq W_j\) for all \(i, j, k\).
\end{definition}

\begin{lemma}
    This is a well-defined notion, namely it does not depend on the choice of decomposition
    \[\beta = \sum_{i,j,k} \beta_{ij}^k \otimes B_{ij}^k.\]
\end{lemma}

\begin{proof}
    The result follows from the fact that being compatible with \(W\) can be defined in a coordinate-free manner. More precisely, we may write
    \[\beta = \sum_{i, j} \eta_{ij}\]
    where
    \[\eta_{ij} \in A^{0, 1}(X, \Hom(\mcF_i, \mcF_j)) \otimes \Hom(V_i, V_j).\]
    We may decompose \(\Hom(V_i, V_j)\) as in \Cref{eq:hom-decomposition}, to write
    \[\eta_{ij} = \eta_{ij}^{WW} + \eta_{ij}^{WU} + \eta_{ij}^{UW} + \eta_{ij}^{UU}.\]
    The condition that \(\beta\) is compatible with \(W\) is equivalent to the condition that \(\eta_{ij}^{WU} = 0\) for all \(i, j\).
\end{proof}

For fixed \(W\), being compatible with \(W\) is a linear condition on \(\beta\). Thus, we may define a linear subspace \(A_W\) of \(A^{0, 1}(X, \End(\mcE_0))\), consisting of those \(\beta\) which are compatible with \(W\).

\begin{lemma}
    \label{lem:kuranishi-slice-compatibility}
    \(v \in H^{0,1}(X, \End(\mcE_0))\) is compatible with \(W\) if and only if \(\Phi(v) - \delbar_0\) is compatible with \(W\).
\end{lemma}

To be able to prove this, we need a better understanding of the Kuranishi map. Referring to \cite[Lemma 6.1]{chenCalabiFlowGeodesic2014} and \cite{kuranishiNewProofExistence1965} for details, we let \(G\) denote the Green's operator for the Laplacian \(\Delta = \delbar_0^*\delbar_0 + \delbar_0\delbar_0^*\), and define
\[F(\beta) = \beta + G\delbar_0^*(\beta \wedge \beta).\]
Observe that \(F(0) = 0\) and that \(\dd F(0) = \id\). Thus, after passing to a Sobolev completion, we may find a local inverse. The map \(\Phi\) is the restriction of this local inverse to \(H^{0, 1}(X, \End(\mcE_0))\).

\begin{proof}
    [Proof of \Cref{lem:kuranishi-slice-compatibility}]

    We first show that if \(\beta\) is compatible with \(W\), then so is \(F(\beta)\). Since being compatible with \(W\) is a linear condition, it suffices to verify that if \(\beta\) is compatible with \(W\), then so is \(G\delbar_0^*(\beta \wedge \beta)\). If we write
    \[\beta = \sum_{i,j,k}\beta_{ij}^k \otimes B_{ij}^k,\]
    then
    \[\beta \wedge \beta = \sum_{i,j,k,p,r}\left(\beta^k_{ij} \wedge \beta^r_{pi}\right) \otimes \left(B_{ij}^k \circ B_{pi}^r\right).\]

    Moreover, as we have a holomorphic and orthogonal decomposition
    \[\mcE_0 = \bigoplus_i \mcF_i \otimes V_i,\]
    it follows that the Green's operator \(G\) and the adjoint \(\delbar_0^*\) only act on the \(A^{0, *}(X, \Hom(\mcF_i, \mcF_j))\) part. Thus,
    \[G\delbar_0^*(\beta \wedge \beta) = \sum_{i,j,k,p,r}\gamma_{ijp}^{kr} \otimes \left(B_{ij}^k \circ B_{pi}^r\right),\]
    for some \(\gamma_{ijp}^{kr} \in A^{0, 1}(X, \Hom(\mcF_i, \mcF_j))\). Thus,
    \[B_{ij}^k \circ B_{pi}^r(W_p) \subseteq B_{ij}^k(W_i) \subseteq W_j.\]
    Thus, \(F(\beta)\) is compatible with \(W\). Moreover, by \Cref{eq:hom-decomposition}, the Sobolev completion of \(A_W\) defines a closed subspace of the Sobolev completion of \(A^{0, 1}(X, \End(\mcE_0))\). With this in mind, we obtain an inverse \(\Phi_W \colon A_W \to A_W\), but this must necessarily be the restriction of \(\Phi\) to \(A_W\). Thus, if \(\beta\) is compatible with \(W\), then so is \(\Phi(\beta)\).
\end{proof}

It remains to show the following.

\begin{lemma}
    \label{lem:kuranishi-slice-compatibility-2}
    \(\beta \in A^{0,1}(X, \End(\mcE_0))\) is compatible with \(W\) if and only if
    \[S = \bigoplus_i F_i \otimes W_i\]
    defines a holomorphic subbundle of \((E, \delbar_0 + \beta)\).
\end{lemma}

\begin{proof}
    \(S\) defines a holomorphic subbundle of \((E, \delbar_0 + \beta)\) if and only if for any \(s \in A^0(X, S)\), we have that \((\delbar_0 + \beta)s \in A^{0, 1}(X, S)\). Since \(S\) is a holomorphic subbundle of \((E, \delbar_0)\) by construction, it is thus equivalent to showing that for any \(s \in A^0(X, S)\), we have that \(\beta s \in A^{0, 1}(X, S)\). After decomposing \(\beta\), we see that this is equivalent to \(\beta\) being compatible with \(W\). 
\end{proof}

\section{Calabi flow}

\label{sec:calabi-flow}

Finally, we will apply \Cref{thm:vector-space} to the Calabi flow. The reduction to the Kuranishi slice is due to Chen-Sun \cite{chenCalabiFlowGeodesic2014}, and we include the statements of the results. The proofs can be found in Chen-Sun \cite{chenCalabiFlowGeodesic2014}, which are closely related to the results for the Yang-Mills flow in the previous section. 

Let \((M, \omega, J_0)\) be a compact K\"ahler manifold. Suppose the class \([\omega]\) is integral. Let \(\mcJ\) denote the space of almost complex structures on \(M\) compatible with \(\omega\), and let \(\mcJ^{\text{int}}\) denote the subspace of integrable almost complex structures. Let \(\mcK = \Ham(M, \omega)\) denote the group of Hamiltonian diffeomorphisms of \((M, \omega)\), which acts on \(\mcJ\) by pullback, keeping \(\mcJ^{\text{int}}\) invariant. The Lie algebra of \(\mcK\) is \(C^\infty_0(M, \R)\). Donaldson \cite{donaldsonRemarksGaugeTheory1997} and Fujiki \cite{fujikiModuliSpaceExtremal1990} showed that the \(\mcK\)-action on \(\mcJ\) has moment map
\[\mu(J) = S(J) - \overline{S},\]
where \(S(J)\) is the hermitian scalar curvature functional, and \(\overline{S}\) is the average scalar curvature.

In this setting, there may not exist a genuine complexification of \(\mcK\), but we can still make sense of ``\(\mcG/\mcK\)'' in this setting. This is given by
\[\mcH = \{\phi \in C^\infty(M, \R) \mid \omega + i\del\delbar\phi > 0\}.\]
Then \(\mcH/\R\) is formally the ``dual'' symmetric space of \(\mcK\).

Analogous to the finite-dimensional setting, we have a functional on \(\mcH\) called the \emph{Mabuchi functional}, which plays the role of the Kempf-Ness functional. It is given variationally by
\[\dd E_\phi(\psi) = -\int_M (S(\phi) - \overline S)\psi \dd\mu_\phi.\]

The norm squared of the moment map is the \emph{Calabi functional}, given by
\[\Cal(J) = \int_M (S(J) - \overline{S})^2 \dd\mu_\omega.\]
Under the natural metric on \(\mcJ\), the gradient of the Calabi functional is given by
\[\grad\Cal(J) = \frac{1}{2}J\mcD_J S(J),\]
where \(\mcD_J\) is the Lichnerowicz operator. The \emph{Calabi flow} in \(\mcJ\) is the gradient flow
\[\dv{t}J(t) = -\frac{1}{2}J(t)\mcD_{J(t)}S(J(t)).\]
If \(J(t)\) is integrable for all \(t\), then as in the finite-dimensional setting, we can lift it to a flow on \(\mcH\), which is given by
\[\dv{t}\phi(t) = S(\phi(t)) - \overline{S}.\]
Our goal will be to study the Calabi flow near a cscK complex structure \(J_0\).

\begin{theorem}
    [{\cite[Theorem 5.3]{chenCalabiFlowGeodesic2014}}] Suppose \(J_0 \in \mcJ\) is cscK. Then there exists \(k \gg 1\), and a \(C^{k, \lambda}\)-neighbourhood \(\mcU\) of \(J_0\) in \(\mcJ^{\text{int}}\), such that the Calabi flow \(J(t)\) starting from any \(J \in \mcU\) exists globally, and converges polynomially fast to a cscK complex structure \(J_\infty \in \mcJ\) in the \(C^{k, \lambda}\)-topology. Up to a Hamiltonian diffeomorphism, we may assume that \(J_\infty\) is smooth. Then the convergence is also in \(C^\infty\).
\end{theorem}

It follows that \((M, J)\) is K-semistable, with a smooth cscK K-polystable degeneration. We also refer to related results in \cite{ortuMomentMapsStability2025}.

Let \(K = \Stab_\mcK(J_0)\) denote the group of Hamiltonian isometries of \((M, \omega, J_0)\), where \(J_0\) is cscK. The assumption that the K\"ahler class \([\omega]\) is integral implies that the \(K\)-invariant inner product on \(\mfk\) is rational \cite{futakiBilinearFormsExtremal1995}. The complexification \(G\) of \(K\) is a subgroup of the group of biholomorphisms of \((M, J_0)\). As \(K\) is a subgroup of \(\mcK\), it acts on \(\mcJ\) with moment map
\[\overline\mu = \pr_\mfk(S - \overline S).\]
We consider the gradient flow of \(\abs{\overline\mu}^2\), which is given by
\[\dv{t}J(t) = -\frac{1}{2}J\mcD_J\overline\mu(J).\]
If we have a solution such that \(J(t)\) is integrable for all \(t\), then we have an associated flow in \(\mcH\), given by
\[\dv{t}\phi(t) = \pr_{f_t^*\mfk}(S(\phi) - \overline S),\]
where \(f_t\) is the family of diffeomorphisms satisfying
\[\dv{t}f_t = -\frac{1}{2}J X_{S(J)}.\]
This is the \emph{reduced Calabi flow} \cite{chenCalabiFlowGeodesic2014}. We now reduce to a finite-dimensional slice. The corresponding elliptic complex is given by

\[\begin{tikzcd}
	{C^\infty_0(M, \C)} & {T_{J_0}\mcJ} & {\Omega^{0,2}(T^{1,0}M)}
	\arrow["{P_\C}", from=1-1, to=1-2]
	\arrow["\delbar", from=1-2, to=1-3]
\end{tikzcd}\]
Let \(H^1\) denote the cohomology of this complex. Then we have the following Kuranishi slice for the \(\mcK\)-action on \(\mcJ\).

\begin{theorem}
    [{\cite[Lemma 6.1]{chenCalabiFlowGeodesic2014}}] 
    \label{thm:kuranishi-slice-calabi}
    There exists a neighbourhood \(B\) of \(0\) in \(H^1\), and a \(K\)-equivariant holomorphic embedding \(\Phi \colon B \to \mcJ\), such that
    \begin{enumerate}[(i)]
        \item \(\Phi(0) = J_0\),
        \item \(\dd\Phi_0 = \id\),
        \item if \(v_1, v_2 \in B\) are in the same \(G\)-orbit, and \(\Phi(v_1)\) is integrable, then \(\Phi(v_2)\) is also integrable, and \(\Phi(v_1), \Phi(v_2)\) are in the same \(\mcK_\C\)-leaf. Conversely, if \(\Phi(v)\) is integrable and \(\dd\Phi_v(u)\) is tangent to the \(\mcK_\C\)-leaf through \(\Phi(v)\), then \(u\) is tangent to the \(G\)-orbit through \(v\).
        \item Any integrable \(J\) in a sufficiently small neighbourhood of \(J_0\) is in the \(\mcK_\C\)-leaf of some \(\Phi(v)\).
    \end{enumerate}
\end{theorem}

Let \(J\) be an integrable almost complex structure sufficiently close to \(J_0\), such that the Calabi flow starting from \(J\) converges to \(J_0\). By (iv) in \Cref{thm:kuranishi-slice-calabi}, we can find \(v \in B\) such that \(J\) is in the \(\mcK_\C\)-leaf of \(\Phi(v)\). In addition, \(0 \in \overline{G \cdot v}\), by the same proof as for \Cref{prop:kuranishi-slice-move-orbit}.

Thus, we may pull back the K\"ahler structure from \(\mcJ\) to \(B\), which we denote by \(\widetilde\Omega\). Moreover, we know that \(K\) acts on \((B, \widetilde\Omega, J)\) holomorphically and isometrically, with moment map \(\widetilde\mu = \Phi^*\overline\mu\). Let \(\widehat\phi(t)\) denote the flow in \(G/K\) corresponding to \(v\).

\begin{lemma}
    [{\cite[Lemma 6.2]{chenCalabiFlowGeodesic2014}}]
    Let \(\widehat\phi(t)\) denote the flow in \(G/K\) corresponding to \(v\), and let \(\phi(t)\) denote the Calabi flow starting from \(J\). Then there exists \(C > 0\) such that for all \(t\),
    \[d(\phi(t), \widehat\phi(t)) \le C.\]
\end{lemma}

\begin{remark}
    The fact that there is an embedding from \(G/K\) to \(\mcH\) is non-trivial, and we refer to \cite{chenCalabiFlowGeodesic2014} for details.
\end{remark}

Thus, we have now proven the following result.

\begin{theorem}
    \label{thm:calabi-flow} Let \(\phi(t)\) denote the Calabi flow starting from \(J\). Then there exists \(v_1, \dots, v_k\), such that
    \[\phi(t) = \exp(i\log(t)v_1 + \dots + i\log\cdots\log(t)v_k) + O(1),\]
    where \(v_1, \dots, v_k\) are commuting elements of \(\mfk\), and \(\exp \colon i\mfk \to G/K\) is the exponential map.
\end{theorem}

\printbibliography

\end{document}